\documentclass[11pt,a4paper,intlimits]{amsart}

\usepackage{amsmath}
\usepackage{esint}
\usepackage{amsthm}
\usepackage{amssymb}
\usepackage{amsfonts}
\usepackage{amsxtra}
\usepackage{color}
\usepackage{enumitem}
\usepackage{pdfsync}
\def\Xint#1{\mathchoice
{\XXint\displaystyle\textstyle{#1}}%
{\XXint\textstyle\scriptstyle{#1}}%
{\XXint\scriptstyle\scriptscriptstyle{#1}}%
{\XXint\scriptscriptstyle\scriptscriptstyle{#1}}%
\!\int}
\def\XXint#1#2#3{{\setbox0=\hbox{$#1{#2#3}{\int}$}
\vcenter{\hbox{$#2#3$}}\kern-.5\wd0}}

\def\dashint{\Xint-}
\numberwithin{equation}{section}

\newtheorem{theorem}{Theorem}[section]
\newtheorem{lemma}[theorem]{Lemma}

\newtheorem{corollary}[theorem]{Corollary}
\newtheorem{definition}[theorem]{Definition}

\newtheorem{remark}[theorem]{Remark}

\newcommand{\re}{\mathbb{R}}
\newcommand{\he}{\mathbb{H}}
\newcommand{\na}{\mathbb{N}}
\newcommand{\setR}{\mathbb{R}}
\newcommand{\setN}{\mathbb{N}}

\DeclareMathOperator{\grad}{grad}
\DeclareMathOperator{\supp}{supp}

\DeclareMathOperator{\dv}{div}

\DeclareMathOperator{\sgn}{sgn}
\DeclareMathOperator{\tv}{TV}
\newcommand{\BV}{BV}
\DeclareMathOperator{\ric}{Ric}

\DeclareMathAlphabet{\mathpzc}{OT1}{pzc}{m}{it}

\DeclareMathOperator{\R}{\mathcal R}
\DeclareMathOperator{\vol}{vol}

\begin{document}

\title[Traces on manifolds applied to conservation laws]{Traces for functions of bounded variation on manifolds with applications to conservation laws on manifolds with boundary}

\author{%
  {%
    Dietmar Kr\"oner, Thomas M\"uller and Lena Maria Strehlau
  }
   \\[1em]%
   \vspace{0.3cm}%
   \\
  {%
    \scriptsize%
    Abteilung f\"ur Angewandte Mathematik, Universit\"at Freiburg,\\
    Hermann-Herder-Str. 10, D-79104 Freiburg, Germany. %
  }\\[1em]%
  {%
    \scriptsize%
    E\lowercase{mails}: \lowercase{\tt dietmar@mathematik.uni-freiburg.de, mueller@mathematik.uni-freiburg.de, lena-strehlau@gmx.de}
   }\\[1em]%
} 

\thanks{The work has been supported by
Deutsche Forschungsgemeinschaft via SFB/TR 71 `Geometric Partial
Differential Equations'. Furthermore, T.M. acknowledges the supported 
by the German National Academic Foundation (Studienstiftung des Deutschen Volkes). }

%\keywords{Hyperbolic; Conservation laws; Riemannian manifolds with boundary; trace theorem}

%\subjclass{35L65; 58J45; 76N10}

\begin{abstract}
In this paper we show existence of a trace for functions of bounded variation on Riemannian manifolds
with boundary.
The trace, which is bounded in $L^\infty$, is reached via $L^1$-convergence and allows an integration by parts formula.
We apply these results in order to show well-posedness and total variation estimates
for the initial boundary value problem for a scalar conservation law on compact Riemannian manifolds with boundary
in the context of functions of bounded variation via the vanishing viscosity method.
The flux function is assumed to be time-dependent and divergence-free.
\end{abstract}

\maketitle

\pagestyle{myheadings}
\thispagestyle{plain}
\markboth{Dietmar Kr\"oner, Thomas M\"uller, Lena Maria Strehlau}{Traces for BV-functions on manifolds with applications to conservation laws}

 \section{Introduction}
%  Machen wir auch eher am Ende, aber das hier sollte ungefh\"ahr rein.
%  \begin{itemize}
%   \item Motivation (Anwendung aus Physik, Meteorologie)
%   \item Problemformulierung und Beschreibung (also sowohl Spursatz als auch Wohlgestelltheit fuer IBVPs), was wir machen.
%   \item Einordnen in bestehende Literatur
%   \item In je einem Satz die folgenden Kapitel beschreiben
%  \end{itemize}
%   
Numerous applications in continuum dynamics are modeled by hyperbolic conservation laws, often posed 
on surfaces or manifolds. Examples are the shallow water equations on a sphere, relativistic flows,
transport processes on interfaces or cell surfaces, just to mention a few.
If the physical domain contains a boundary or only a part of a larger closed manifold is of interest,
one has to consider initial boundary value problems on manifolds with boundary.

In this work we will show existence of a trace for BV functions on Riemannian manifolds with
boundary and conclude some properties in order to
prove existence and uniqueness in the space of BV functions, 
and total variation estimates for a solution $u:\bar M\times(0,T)\rightarrow\setR$ 
of the following problem:
\begin{align}
\partial_t u + \dv_g(f(u,x,t))&=0\ \ \text{in $M \times (0,T)$},\label{11}\\
u(\cdot, 0)&=u_0 \   \text{in $M$},\label{12}\\
u&=0 \quad \text{on the outflow part (cf. (\ref{outflow})) of } \partial M ,\label{13}
\end{align}
where $\bar M$ is a smooth, compact manifold with boundary and smooth Riemannian metric $g$, $M=\bar M\backslash \partial M$,
$f(u, \cdot, t)$ a family of smooth vector fields on $\bar M$ parametrized smoothly by
 $u \in \re$ and $t \in [0,T]$ 
with $|\partial_u f|_g \in L^\infty(\re \times M\times (0,T))$, $ (\dv_g f)(u,x,t)=0$ and $u_0 \in BV(M)$.\\

Let us briefly summarize related results. 
The theory on BV functions and their traces is well established for Euclidean domains.
For a comprehensive introduction to this subject we refer to \cite{Ambrosio,Evans/Gariepy,Giu84}.
Studying heat semigroups on Riemannian manifolds, parts of the theory for BV functions
have been generalized to Riemannian manifolds without boundary in \cite{Miranda}.
In particular, the heat semigroup allows an appropriate mollification of BV-functions in this setting.
BV functions in a more general understanding are studied by Vittone \cite{Vit12}.
He shows existence and some properties of a trace operator for BV functions defined on 
%an $X$-Lipschitz 
a special domain with compact boundary in Carnot-Carath\'eodory spaces
and where the bounded variation is defined with respect to a family of vector fields.

The theory on BV functions plays a central role in the study of conservation laws in Euclidean space.
Indeed, they form a natural solution space as the property of a bounded variation is conserved for scalar equations 
with respect to time evolution, and numerical analysis for these PDEs heavily relies on this regularity.
For material on conservation laws in the Euclidean case we refer to \cite{Daf00} and the references therein.
The initial boundary value problem for the Euclidean case was first solved by Bardos et al. \cite{BardosLeroux} exploiting
the fact that functions of bounded variation admit traces on the boundary. It was Otto \cite{Otto}, who established 
$L^\infty$-theory for the initial boundary value problem.
To our knowledge, conservation laws on manifolds were studied for the first time by \cite{Pan97} showing
existence and uniqueness for a geometry-independent variant of \eqref{11}.
For the case of compact manifolds without boundary
the theoretical and numerical study of conservation laws on manifolds has reached
significant progress \cite{BL07,ABL05, BFL09, ALN11,LON09,Dziuk/Kroener,MuellerLengeler,Gie09,GM14}
during the last decade. The Dirichlet problem for a geometry-independent formulation of
\eqref{11} was addressed by Panov using a kinetic formulation.\\

Our results on traces as well as on well-posedness for the initial boundary value problem seem to be new
in the context of Riemannian manifolds with boundary.
Notation and preliminaries are presented in Section \ref{sec:notation}.
In Section \ref{sec:traces} we show existence of a trace via a partition of unity.
This definition, based on local terms, turns out to be well-posed
and thus, independent of the choice of coordinates. A key result for the application to conservation laws is
the partial integration formula \eqref{trace}, which also guarantees uniqueness of the trace.
Section \ref{sec:ibvp} is devoted to the study of the initial boundary value problem. To this end
a parabolic regularization of \eqref{11} is considered, adding small viscosity.
Estimates for the solution of the regularized problem, uniform in the viscosity parameter, guarantee convergence of 
a subsequence to the entropy solution. Uniqueness is proved by transferring Kruzkov's doubling of variables to our setting.

 \section{Notation and preliminaries}
 \label{sec:notation}
 %Machen wir fortlaufend.
% Hier kommt im Wesentlichen die Notation rein. Also vielleicht Mannigfaltigkeit, Tangentialraum, Vektor- und Tensorfelder, Riemannsche Metrik, %Skalarprodukt, Zusammenhang/Kovariante Ableitung, Abstand,
% Paralleltransport, Geodaetische, Exponentialabbildung, Normalkoordinaten, Mannigfaltigkeit mit Rand, Integration, Lebesgue- und %Sobolevraeume, Glaettungskern und Eigenschaften und was wir sonst noch so brauchen.
% Ich f\"ande es gut, wenn wir uns, was die Diffgeobegriffe betrifft, relativ nah an der Arbeit von Daniel und mir halten.

In this section we give a short overview on Riemannian geometry, for a comprehensive introduction see e.g. \cite{doCarmo,LeeRM}.

Throughout the whole paper, let $\bar M$ be an $n$-dimensional, compact, oriented, smooth manifold with boundary $\partial M$. The inner of $\bar M$,
which is a manifold without boundary, we denote with $M:= \bar M \backslash \partial M$. Let $g := \langle\cdot, \cdot\rangle_g$ be a smooth Riemannian 
metric defined on $\bar M$ and  $\nabla^g= \nabla$ the associated Levi-Civita connection on the tangential bundle $T \bar M$. With $\Gamma(T \bar M)$ we denote the set of
differentiable vector fields on $\bar M$ and with $\Gamma_0(TM)$ the set of differentiable vector fields on $M$ of compact support contained in $M$.
The space of smooth $(r,s)$-tensor fields is denoted by $\Gamma(\mathcal T_s^r M)$. 
We call $(\bar M, g)$ a Riemannian manifold.
The pair $(\partial M, \widetilde g)$ where $\tilde g$ is the $g$-induced metric on $\partial M$, is an $(n-1)$-dimensional Riemannian manifold. The Riemannian distance
of two points $x,y\in\bar M$ will be denoted by $d_g(x,y)$ and the geodesic ball around $x$ with radius $\rho$ by  $B_\rho^g(x)$.
Using Einstein's summation convention we write for the scalar product of two tangential vectors $X,Y$ locally $\langle X,Y\rangle_g = g_{ij}X^iY^j$.
The summation convention will be used throughout the whole paper. For $(r,s)$-tensors $F,G$ we define
\begin{align*}
 \langle F, G\rangle_g := g^{i_1 k_1}  \dots g^{i_rk_r}  g_{j_1l_1}  \dots g_{j_sl_s}  F_{i_1\dots i_r}^{j_1\dots j_s} G_{k_1\dots k_r}^{l_1\dots l_s}.
\end{align*}
For the $g$-induced norm of a tensor $F$ we write $|F|_g:= \langle F,F \rangle_g^{\frac{1}{2}}$.
In local coordinates we write $X^i_{\, ;j}:= \partial_j X^i+X^k \Gamma_{jk}^i$ for the covariant derivative of a vector field $X$ with the Christoffel symbols $\Gamma_{jk}^i$.
By $\nabla^k$ we denote the $k$-fold application of $\nabla$.
We can associate each covector field with a vector field by lowering the index. E.g., the covariant vector field $\nabla u$ can be associated
with the vector field $\grad_g u$ by 
$(\nabla u)_i= u_{;i}=\partial_i u=g_{ij} g^{jk}\partial_k u=  g_{ij} (\grad_g u)^j= g_{ij}u_{;}^{\, j}$.
A generalization of the connection $\nabla$ for $(r,s)$-tensors $A \in  \Gamma(\mathcal T_s^r M)$  is given by
\begin{align*}
   \nabla: \Gamma(TM) \times \Gamma(\mathcal T_s^r M) &\rightarrow \Gamma(\mathcal T_s^r M)\\
  (X, A) &\mapsto  \nabla_X A
 \end{align*}
 with
 \begin{align*}
 (\nabla_X A)(\omega^1, \dots, \omega^r, Y_1, \dots Y_s):= &X(A(\omega^1, \dots, \omega^r, Y_1, \dots Y_s))\\
 &-\sum_{j=1}^r A (\omega^1, \dots, \nabla_X \omega^j, \dots, \omega^r, Y_1, \dots Y_s)\\
 &- \sum_{i=1}^s A(\omega^1, \dots, \omega^r, Y_1,\dots, \nabla_X Y_i \dots Y_s).
\end{align*}

For a smooth vector field $X$, we define $\dv_gX$ by
\begin{align*}
\int_M u \,\dv_g X\, dv_g = -\int_M \langle \grad_g u, X \rangle_g \, dv_g \, \forall u \in C_0^\infty(M)
\end{align*}
  where locally $dv_g =\sqrt{|g|} \ dz$ denotes the Riemannian volume element
  for positively oriented coordinates $z \in \re^n$ and $|g|:=\det(g_{ij})$.
In local coordinates this yields 
$\dv_g X= X^i_{\,; i}= \frac{1}{\sqrt{|g|}}\partial_i( X^i\sqrt{|g|})$. 
For an arbitrary $(r,0)$-tensor $\alpha$ we define
\begin{align*}
 \dv_g(\alpha)_{i_1\dots i_{r-1}}:= g^{jl}\nabla_j \alpha_{li_1 \dots i_{r-1}}= \nabla^j \alpha_{ji_1 \dots i_{r-1}}
\end{align*}
with $\nabla_j = \nabla_{\partial_j}$.
We define the Laplace-Beltrami $\Delta_gu$ for $u \in C^\infty(M)$ as
\begin{align*}
\Delta_g u := \dv_g(\grad_g u) 
\end{align*}
and in general
\begin{align*}
 \Delta_g =\dv_g \nabla=\text{trace} \nabla^2.
\end{align*}
For $u\in C^\infty(M)$ we define the commutator
\begin{align*}
 [\Delta_g, \nabla]u:= \Delta_g \nabla u- \nabla\Delta_g u.
\end{align*}
Let $\R$ denote the Riemannian curvature tensor, i.e.
\begin{align*}
 \R(X,Y)Z:= \nabla_X \nabla_Y Z- \nabla_Y \nabla_X Z-\nabla_{[X,Y]}Z
\end{align*}
for $X,Y,Z \in \Gamma(TM)$ and $\ric$ the Ricci tensor defined in local coordinates as
\begin{align*}
 \ric_{jk}:= \R_{ijk}^{\,\,\,\,\,\,\, i}.
\end{align*}
One can prove (cf. \cite{MuellerLengeler}) that in local coordinates
\begin{align*}
\ric_{jk} u_{;}^{\,k}= \R_{ijk}^{\,\,\,\,\,\,\, i}\, u_;^{\, k}= u_{;\, \, ji}^{\,i}-u_{;\, \, ij}^{\,i}
\end{align*}
and thus
\begin{align}\label{eq:commutator}
 [\Delta_g, \nabla]u= \Delta_g \nabla u- \nabla \Delta_g u= \ric(\nabla u, \cdot).
\end{align}
%where $\Ric$ and $\R$ denote the Ricci and the Riemannian curvature tensor respectively.
By $\vol_g(U):= \int_U 1 \, dv_g$  for a subset $U \subset M$ we define the volume measure on $M$ and, analogously,
by $\vol_{\widetilde g}$ the volume measure on $\partial M$.\\

At several points we will work with local coordinates. Note that a chart $\varphi:U \rightarrow V$ maps a portion $U \subset \bar M$ into the half space 
$\he^n := \{ (z^1,..., z^n) \in \re^n|\, z^n \geq 0 \}$. Of special use are Riemannian normal coordinates centered at some point $x \in U \subset M$ which are given by a chart $\varphi:U\mapsto\varphi(U)\subset\setR^n$
defined as the concatenation of the inverse of the exponential map $\exp_x$ with the isomorphism from the tangent space $T_xM$ to $\re^n$
induced by choosing a $g(x)-$orthonormal basis of $T_xM$.
In these coordinates we obtain in $\varphi(x)=(0, \dots, 0)$
\begin{align*}
g_{ij}=g^{ij}= \delta_{ij}, \, \partial_k g_{ij}=0\text{ and } \Gamma_{ij}^k =0.
\end{align*}
Furthermore, for every point $x \in \partial M$ there is a neighborhood $U \subset \bar M$ on which we can define geodesic boundary coordinates centered 
in $x \in \partial M$ (see \cite{Hsiung}).
On $\partial M\cap U$ we define Riemannian normal coordinates $(x^1, \dots, x^{n-1})$ related to $\tilde g$
and extend them by $x^n$ such that the $x^n$-curve is a geodesic on $M$ which is orthogonal to the $x^i$-curves.
We obtain for these coordinates
\begin{align*}
g_{nn}=g^{nn}=1 \text{ and } g_{ni}=g^{ni}=0 \text{ for } i=1, \dots, n-1 \quad \text{in }U
\end{align*}
and as immediate implications for $\tilde g$ and the unit outer normal $N$ on $\partial M \cap U$
\begin{align*}
N^n&=-1, \\
N^i&=0 \text{ for } i=1, \dots, n-1\text{ and }\\
\sqrt{|g|}&=\sqrt{|\widetilde g|}.
\end{align*}

\begin{lemma} \label{lemSaks}
Let $(\bar M,g)$ be a compact, oriented, smooth Riemannian manifold with boundary.
Then for $u\in C^\infty(\bar M)$ we have
 \begin{align*}
  \lim_{\eta \searrow 0} \int_{ \{x \in M |\, |u(x)| < \eta\}}  | \grad_g u(x)|_g \, dv_g =0.
 \end{align*}
 \end{lemma}
 \begin{proof}
 The claim follows from \cite[p. 1020, Lemma 2]{BardosLeroux}  via a partition of unity.
 \end{proof}

\begin{lemma} [Lebesgue's Theorem on manifolds] \label{lemLebesgueM}
 Let $ u\in L^1(M)$. Then we have for almost every $x \in M$

\begin{align}
&\lim_{\rho\searrow 0} \dashint_{B_\rho^g(x)} |u(x)-u(y)|\,  dv_g(y) =0\, \text{and} \label{lebesgue1}\\
&\lim_{\rho \searrow 0} \dashint_{B_\rho^g(x)} u(y) \, dv_g(y) =u(x). \label{lebesgue2}
\end{align}

\end{lemma}
\begin{proof}
The proof for the Euclidean case can be repeated after a partition of unity.
%(\ref{lebesgue1}) and (\ref{lebesgue2}) follow directly from the Euclidean case via 
\end{proof}
At several places we will make use of functions 
$R_\delta: \bar M \xrightarrow {C^\infty} [0,1]$ for $\delta>0$
that are only supported in a small neighborhood of $\partial M$, precisely we 
require
\begin{equation}\label{rhodelta}
 \begin{aligned}
R_\delta &\equiv 0\, \text{ in $M_\delta:= \{x \in \bar M|\,  d_g(x,\partial M) > \delta\}$},\\
R_\delta &\equiv 1 \, \text{ in $ \{ x \in \bar M | \, d_g(x, \partial M) \leq \frac{\delta}{2}\}$}.
\end{aligned}
\end{equation}
\begin{lemma}\label{lemRadonMeasure}
 For a positive finite measure $\mu$ on $M$ we have 
 \begin{align}
  \lim_{\delta\searrow 0} \int_M R_\delta \ d\mu  = 0
 \end{align}
\end{lemma}
\begin{proof}
 Without loss of generality we can consider a monotone, positive sequence $(\delta_n)_{n\in\setN}$ 
 with $\delta_n\searrow 0$ for $n\rightarrow \infty$. 
 For $A_n:= M\backslash M_{\delta_n}$ and the disjoint sets $B_n:= A_{n}\backslash A_{n+1}$
 we obtain from the $\sigma$-additivity and finiteness of $\mu$ 
 \begin{equation}
  \sum_{l=1}^{k} \mu(B_l) = \mu(A_1) - \mu(A_{k}) \leq \mu(A_1) <\infty.
 \end{equation}
The positivity of $\mu$ yields convergence of the sequence $\sum_{l=1}^{k} \mu(B_l)$
towards $\mu(A_1)$ for $k\rightarrow\infty$
and hence
\begin{equation*}
 \lim_{k\rightarrow\infty}\mu(A_k)=\mu(A_1)-\lim_{k\rightarrow\infty}\sum_{l=1}^{k} \mu(B_l) =0
\end{equation*}
which completes the proof.
\end{proof}

% To prove uniqueness we will need the following two Riemannian geometric results which are also given in \cite{MuellerLengeler}.
% 
% \begin{lemma} \label{lemdxy}
% Let  $x,y \in M$ with $d(x,y)$ sufficiently small  and  $\gamma_{xy}: [0,d_g(x,y)] \rightarrow M$ be eine nach BOGENLAENGE parametrisierte Geodätische von $x$ nach $y$ ,then we have
% \begin{align}
% \grad_g^xd(x,y)&= - \dot \gamma_{xy}(0) \text{ and} \label{lemdxy1}\\
% \grad_g^yd(x,y)&= \dot \gamma_{xy}(d(x,y)).\label{lemdxy2}
% \end{align}
% \end{lemma}
% 
% \begin{lemma} \label{lemparakov}
% Let  $x,y\in M$ and $\gamma$ be the geodesic connection between $x$ and  $y$.
% Then we have for  $f \in \Gamma(TM)$:
% \begin{align*}
% f(y)-P_{xy}f(x)=-\nabla_{\dot \gamma_{yx}(0)}f(y) d(x,y)+o(d(x,y)).
% \end{align*}
% 
% \end{lemma}

%%%%%%%%%%%%%%%%%%%%%%%%%%%%%%%%%%%%%%%%%%%%%%%%%%%

 \section{Traces for functions of bounded variation on manifolds}
\label{sec:traces}
 
In this section we will show the existence and fundamental properties of traces  for functions of bounded variation on manifolds. The key result of this section is Theorem
\ref{SpursatzM}.
Analogous results for the Euclidean case are given in \cite[pp. 176-183]{Evans/Gariepy}.\\
%Throughout this section let $\bar M$ be a smooth manifold with boundary $\partial M$ and $M:= \bar M \backslash \partial M$.

\begin{definition}[$\BV$ functions on manifolds]
The total variation of a function  $u \in L^1(M)$ on $M$ is defined as 
\begin{align*}
\tv(u,M):= \sup\left\{ \int_M u\, \dv_g X \, dv_g| \, X \in \Gamma_0(T M), |X|_g \leq 1\right\}.
\end{align*}
For smooth functions $u:M \rightarrow \re$ we have
\begin{align*}
\tv(u,M):= \int_M |\grad_g u|_g \, dv_g.
\end{align*}
We define the set of functions of bounded variation on $M$ as

\begin{align*}
\BV(M):= \{ u \in L^1(M)| \tv(u,M) < \infty\}.
\end{align*}

\end{definition}
For the proof of Theorem \ref{SpursatzM} we will use the notation of the following Lemma from \cite[p. 167, Theorem 1]{Evans/Gariepy}.

\begin{lemma} \label{thmStructure}
 Let  $V\subset \re^n$ be open, $h \in \BV_{\text{loc}}(V)$. Then there exist a Radon measure $|Dh|$ on $V$ and a 
 $|Dh|$-measurable function  $\sigma_h:V \rightarrow \re^n$, such that
 \begin{align*}
         |\sigma_h|&=1\,\ \text{ $|Dh|$-almost everywhere and }\\
% \end{align*}
% \begin{align*}
   \int_V h \, \dv\phi\, dz&= - \int_V \langle \phi, \sigma_h\rangle \, d|Dh|\text{ for all $\phi \in C_0^\infty(V, \re^n)$},
  \end{align*}
  where $\langle\cdot,\cdot\rangle$ denotes the standard Euclidean scalar product in $\setR^n$.
\end{lemma}

%%%%%%%%%%%%%%%%%%%%%%%%%%%%%%%%%%%%%%%%%%%%%%%
 An analogous result for manifolds  is given in \cite[p. 104]{Miranda}:

\begin{lemma} \label{mppp}
For a function $u \in \BV(M)$ there exist a finite measure  $|Du|$ on $M$ and a $|Du|$-measurable function
$\sigma_u:M \rightarrow TM$ such that
 
\begin{align*}
 |\sigma_u|_g=1\, \text{ $|Du|$-almost everywhere and}
\end{align*}
\begin{align*}
\int_M u \, \dv_g X\, dv_g = - \int_M \langle \sigma_u, X\rangle_g \, d|Du|\text{ for all $X\in \Gamma_0(T M)$.}
\end{align*}

\end{lemma}

%%%%%%%%%%%%%%%%%%%%%%%%%%%%%%%%%%%%%%%%%%%%%

\begin{theorem} [Traces for functions of bounded variation on manifolds] \label{SpursatzM}
There exists a unique linear operator  
\begin{align*}
 T:\BV(M)\rightarrow L^1(\partial M; dv_{\widetilde g})
\end{align*}
where $dv_{\widetilde g}$ denotes the $(n-1)$-dimensional Riemannian volume element on $\partial M$ such that
\begin{align}
\int_M u\, \dv_g \, X \, dv_g= -\int_M \langle X , \sigma_u \rangle_g \, d|Du|+ \int_{\partial M} \langle X, N \rangle_g\, Tu \, dv_{\widetilde g} \label{trace}
\end{align}
for all $u\in \BV(M)\cap L^\infty(M)$ and all $X \in \Gamma(T\bar M)$, where $|Du|$ and $\sigma_u$ are defined as
in  Lemma \ref{mppp} and $N$ denotes the unit outer normal.

\end{theorem}

%\com{Begin \"Anderungsvorschlag}
\begin{proof}
Let $u\in \BV(M)$ and $X \in \Gamma(T\bar M)$. 
We can write
\begin{align}
\int_M u\, \dv_g X \, dv_g= \int_M u\, \dv_g (XR_\delta) \, dv_g+
\int_M u\, \dv_g (X(1-R_\delta)) \, dv_g \label{T1T2}
\end{align}
with $R_\delta$ as in (\ref{rhodelta}).
The application of Lemma \ref{mppp} yields existence of a finite measure $|Du|$ and a $|Du|$-measurable function  $\sigma_u$ 
such that
\begin{align*}
\int_M u\, \dv_g (X(1-R_\delta)) \, dv_g&= - \int_M \langle \sigma_u, X (1-R_\delta)\rangle_g \, d|Du|\\
&\xrightarrow{\delta \searrow 0} - \int_M \langle \sigma_u, X \rangle_g \, d|Du|.
\end{align*}
since $X (1-R_\delta) \in \Gamma_0(T M)$ and by the use of Lemma \ref{lemRadonMeasure}.

Considering the first term on the right hand side of \eqref{T1T2} we introduce 
a finite collection of charts $\{ (U_i, \varphi_i)\}_{i \in I}$ in geodesic boundary coordinates which 
covers $ M \backslash M_\delta $ and a subordinate partition of unity $\{\psi_i\}_{i \in I}$
such that 
\begin{align}
\int_{M} u \, \dv_g (X R_\delta)\, dv_g= \sum_{ i \in I} \int_{\varphi_i(U_i)}(\psi_i  u \, \dv_g (X R_\delta))
\circ \varphi_i^{-1} \sqrt{|g_i|}\, dz.\label{partition}
\end{align}

From \cite[p. 177, Theorem 1] {Evans/Gariepy} we know that for  $V \subset \re^n$ open and bounded with $\partial V$ Lipschitz there exists  a linear trace operator
\begin{align*}
\Theta: \BV(V) \rightarrow L^1(\partial V; H^{n-1}),
\end{align*}
where $H^{n-1}$ denotes the Hausdorff-measure restricted to $\partial V$, such that
\begin{align}
\int_V h \dv\phi\, dz=- \int_V\langle \phi, \sigma_h \rangle \, d|Dh|+ \int_{\partial V}\langle \phi, \nu\rangle\, \Theta h \, dH^{n-1}\label{traceR}
\end{align}
for all $h \in \BV(V) $ and $\phi \in C^\infty(V, \re^n)$, where $|Dh|$ and $\sigma_h$ are defined as in Lemma \ref{mppp}
and $\nu$ denotes the unit outer normal to $\partial V$ with respect to the standard Euclidean scalar product
$\langle\cdot,\cdot\rangle$.

We want to apply (\ref{traceR}) to an arbitrary summand of the right-hand side of (\ref{partition}) and suppress
the index $i$ in the following, i.e. $\varphi_i= \varphi: U \rightarrow \varphi(U)= V$ etc.
It is easy to see that \begin{align}
 \bar u:= (u\psi)\circ \varphi^{-1} \in \BV(V)
 \end{align} 
as on a compact set, i.e. particularly on the set $\supp(\psi \circ \varphi^{-1})\subset \setR^n$, there exist constants $c, c' \in \re^+$ such that
\begin{align*}
\frac{1}{c} \leq \sqrt{|g|}\leq c
\text{ and }\frac{1}{c'}|z|\leq |z|_g := \sqrt{g_{ij} z^i z^j} \leq c'|z|
\end{align*}
uniformly for $z \in \re^n$.
After the introduction of $\phi_\delta:=(\phi^1_\delta,\ldots,\phi^n_\delta)$
with $\phi^i_\delta:=\bar X^i (R_\delta\circ \varphi^{-1}) \sqrt{|g|}$
and $\bar X^i$ being the local components of $X$ the application of \eqref{traceR} yields
\begin{align}\label{SpuraufV}
\int_{V}(\psi  u \, \dv_g (X R_\delta))\circ \varphi^{-1}\sqrt{|g|} \, dz
 =\int_V\bar u\, \dv\phi_\delta\, dz = S_1^\delta + S_2^\delta + S_3^\delta 
\end{align}
%  = -\int_V \langle \phi_\delta, \sigma_{\bar u}\rangle \, d|D\bar u|+ \int_{\partial V}\langle \phi_\delta, \nu\rangle 
%  \, \Theta \bar u \, dH^{n-1}\notag \\
%  =\underbrace{-\int_V \langle \phi_\delta , \sigma_{\bar u}\rangle \, d|D\bar u|}_{=:S_1^\delta}+
%  \underbrace{\int_{\partial V\cap \partial \he^{n}}\langle \phi_\delta , \nu\rangle \, \Theta \bar u \, dH^{n-1}}_{=:S_2^\delta}
%  +\underbrace{\int_{\partial V\backslash \partial \he^{n}}\langle \phi_\delta, \nu\rangle \, \Theta \bar u \, dH^{n-1}}_{=:S_3^\delta}. 
with
\begin{align*}
S_1^\delta&:=-\int_V \langle \phi_\delta , \sigma_{\bar u}\rangle \, d|D\bar u|,\\
S_2^\delta&:=\int_{\partial V\cap \partial \he^{n}}\langle \phi_\delta , \nu\rangle \, \Theta \bar u \, dH^{n-1},\\
S_3^\delta&:=\int_{\partial V\backslash \partial \he^{n}}\langle \phi_\delta, \nu\rangle \, \Theta \bar u \, dH^{n-1}.
\end{align*}
The terms $S_1^\delta$ and $S_3^\delta$ converge to zero for $\delta \searrow 0$ which can be seen
by Lemma \ref{lemRadonMeasure} and the fact that 
$\Theta \bar u \in L^\infty(\partial V \backslash \partial \he^n) $ (see \cite[p. 181, Theorem 2]{Evans/Gariepy}).
% With $\phi_\delta=\bar X(R_\delta \circ \varphi^{-1}) \sqrt{|g|}$:
% \begin{align*}
% |S_1^\delta|&=  \left|\int_V \langle \phi_\delta , \sigma_{\bar u}\rangle \, d|D\bar u|\right|\\
% &=\left|\int_V \langle \bar X (R_\delta\circ \varphi^{-1}) \sqrt{|g|}, \sigma_{\bar u}\rangle
%  \, d|D\bar u|\, \right|\\
%  & \leq \left\|\bar X \,  \sqrt{|g|}\right\|_{L^\infty(V)} |D \bar u| \left(V \cap \text{ supp}(R_\delta \circ \varphi^{-1})\right)\\
%  &\xrightarrow {\delta \rightarrow 0} 0,
% \end{align*}
% which follows of the $\sigma$-additivity of the measure $|D \bar u|$.\\

% \begin{align*}
% |S_3^\delta|= &\left|\int_{\partial V\backslash \partial \he^{n}}\langle \phi_\delta, \nu\rangle \, \Theta\bar u \, dH^{n-1}\right|\\
%  =&\left|\int_{\partial V\backslash \partial \he^{n}}\langle \bar X (R_\delta\circ \varphi^{-1}) \sqrt{|g|}, \nu\rangle \, \Theta \bar u \, 
% dH^{n-1} \right|\\
% & \xrightarrow {\delta \rightarrow 0}\, 0,
% \end{align*}
% 
% because $\Theta \bar u \in L^\infty(\partial V \backslash \partial \he^n) $ (see \cite{Evans/Gariepy}, p. 181, Theorem 2).\\

%For  $S_2^\delta$, with $R_\delta \equiv 1$ on $\partial M$, the equality

Regarding $S_2^\delta$ recall that we chose geodesic boundary coordinates and hence
$\langle \phi_\delta, \nu\rangle = -\bar X^n  \sqrt{|g|} = \langle X, N\rangle_g\circ\varphi^{-1} \sqrt{|\tilde g|}$
on $\partial V\cap \partial\he^n$.
Consequently
\begin{align}
S_2^\delta= \int_{\partial V\cap \partial \he^{n}}\langle \phi_\delta, \nu\rangle \, \Theta \bar u \, dH^{n-1}
 = \int_{\partial U\cap \partial M} \langle X, N\rangle_g\,  T(u\psi) \,
 dv_{\widetilde g} \label{untererRand}
\end{align}
with $T(u\psi):= \Theta(\bar u) \circ \varphi : \partial U \cap \partial M \rightarrow \re$.
Analogously, we define the trace for each $u\psi_i$ on  $\partial M$ as
\begin{align*}
T^iu:=  \begin{cases}
\Theta ((u \psi_i)\circ \varphi_i^{-1})\circ \varphi_i &\text{ on } \partial M \cap \partial U_i, \\
0 &\text{ on } \partial M \backslash \partial U_i 
\end{cases}
\end{align*}
and
\begin{align*}
 Tu := \sum_{i\in I} T^iu.
\end{align*}
For $\delta\searrow 0$ in \eqref{partition} we finally obtain
\begin{align}
 \lim_{\delta \searrow 0} \int_{M } u \, \dv_g (X R_\delta)\, dv_g = \int_{\partial M} \langle X, N\rangle_g\, Tu \, dv_{\widetilde g}
 \label{KonvergenzSpur}
\end{align}
which proves existence.\\
To prove uniqueness we assume that (\ref{trace}) holds for $Tu$ and for $v \in L^\infty(\partial M)$. Subtraction of the corresponding equations yields

\begin{align*}
 \int_{\partial M} \langle X, N\rangle_g\, (Tu-v) \, dv_{\widetilde g} =0 \, \text{ for all $ X \in \Gamma(T\bar M)$}.
\end{align*}
This is true particularly in the limit $\delta\searrow0$ 
for $X=X_\delta:= R_\delta \phi N_\delta$ with an arbitrary $\phi \in C^\infty(\bar M)$ 
and $N_\delta$ being an extension of $N$ to $\bar M\backslash M_\delta$ for small $\delta$.
The fundamental lemma of calculus of variations proves uniqueness  up to a set of $\text{vol}_{\widetilde g}$-measure zero.
Linearity can be proved by considering \eqref{trace} for functions $u,v\in BV(M)$ and their sum $u+v$ for $X=X_\delta$ in the limit
$\delta\searrow0$.
\end{proof}

From the proof of Theorem \ref{SpursatzM} we obtain the following corollary.

\begin{corollary}\label{kordirFolgerungSpur}
 For $u \in \BV(M)\cap L^\infty(M)$ and $X \in \Gamma(TM)$ we have

 \begin{align*}
  \lim_{\delta \searrow 0} \int_M u \langle \grad_g R_\delta, X\rangle_g \, dv_g = \int_{\partial M} TU \langle X, N \rangle_g \, dv_{\widetilde g}.
 \end{align*}
\begin{proof}

 \begin{align*}
  \int_M u \langle \grad_g R_\delta, X\rangle_g \, dv_g
  &=\int_M u \, \dv_g (X\, R_\delta)\, dv_g-
 \underbrace{ \int_M u R_\delta \dv_g X\, \, dv_g}_{\xrightarrow{\delta \searrow 0} 0}\\
 &\xrightarrow{\delta \searrow 0} \int_{\partial M} Tu \langle X, N\rangle_g \,  dv_{\widetilde g}\quad (\text{cf. (\ref{KonvergenzSpur})}).
 \end{align*}

\end{proof}

\end{corollary}

%%%%%%%%%%%%%%%%%%%%%%%%%%%%%%%%%%%%%%%%%%%%%%
\begin{corollary}[Properties of the trace  $Tu$] \label{korEigenschaftenTu}
The trace $Tu: \partial M \rightarrow  \re$ satisfies
\begin{enumerate}
  \item For $vol_{\widetilde g}$-almost every $x_0 \in \partial M$ we have
  \begin{align*}
   \lim_{\rho \searrow 0}\dashint_{B_\rho^g(x_0)} |u -Tu(x_0)|\, dv_g =0
  \end{align*}
and
\begin{align*}
 Tu(x_0)= \lim_{\rho \searrow 0} \dashint_{B_\rho^g(x_0)} u \, dv_g
\end{align*}
 
\item $Tu \in L^\infty (\partial M)$.

\item \label{CorTraceComm}For  $ h\in C^1\left([-z, z]\right)$ with  $[-z,z] \subset \re$ and $z> \|Tu\|_{L^\infty(\partial M)}$ it is
\begin{align*}
T[h(u)]= h(Tu)  
\end{align*}
almost everywhere on $\partial M$.
\end{enumerate}

\end{corollary}

\begin{proof}
To prove claim (1), we choose normal coordinates on $B_\rho^g(x_0)$ and refer to the Euclidean
case (\cite[p. 181, Theorem 2]{Evans/Gariepy}).
Claim (2) follows immediately from claim (1).
The mean value theorem together with claim (1) yield claim (3).
The argumentation is analogous to the one in the Euclidean case.
\end{proof}

%%%%%%%%%%%%%%%%%%%%%%%%%%%%%%%%%%%%%%%%%%%%%%%%%%%%%%%%%%%

Considering $M \times (0,T)$ as an $(n+1)$-dimensional manifold endowed with the Riemannian metric $g_T$, locally
defined by $g_T=dt^2+g_{ij}dx^i dx^j$, we will need the following Lemma.

\begin{lemma} \label{lemProduktBV}
For $ u \in \BV(M \times (0,T))$ we have 
\begin{enumerate}
\item  $ u(\cdot, t) \in \BV(M)$ for almost every  $ t \in (0,T)$ and 
\item  $u(x, \cdot) \in \BV((0,T))$ for almost every $ x \in M$.
\end{enumerate}
\end{lemma}

\begin{proof}
This can be proved via a partition of unity and by then applying  Lemma 1 from \cite[p. 1019]{BardosLeroux}.
\end{proof}

%%%%%%%%%%%%%%%%%%%%%%%%%%%%%%%%%%%%%%%%%%%%%%%%%%%%%

\begin{lemma}\label{korfBVM}
Let $u \in \BV(M)\cap L^\infty(M) $ and $F\in C^1(\re \times \bar M)$. Then the function
\begin{align*}
  x \mapsto F(u(x), x)
\end{align*}
is in $\BV(M)$. 
\end{lemma}

\begin{proof}
For the proof we use Proposition 1.4, Theorem 2.1 and Theorem 3.3  from \cite[p. 105, p. 109 and p. 117]{Miranda},
where the  existence of a sequence 
$(u_j)_{j \in \na} \in C_0^\infty(M)$ with the properties 
\begin{align*}
 u_j &\rightarrow u \, \text{ in $L^1( M)$, } \\
 \tv(u,M) &= \lim_{j \rightarrow \infty} \int_{ M} |\grad_g u_j | \, dv_g \text{ and}\\
\|u_j \|_{L^\infty(M)} &\leq \|u\|_{L^\infty(M)} \text{ for all $j \in \na$}
\end{align*}
is shown. Using the fact that
\begin{align*}
 F(u_j, \cdot) \rightarrow F(u, \cdot) \, \text{ in $L^1(M)$}
\end{align*}
and the boundedness of $\| \partial_u F(u_j (\cdot), \cdot)\| _{L^{\infty}(M)}$ and $\| (\grad_g F) (u_j, \cdot) \|_{L^\infty(M)}$ uniformly in $j$
we obtain

\begin{align*}
\sup\left\{ \int_M F(u,\cdot) \, \dv_g X \, dv_g\, | \, X \in \Gamma_0(T M), \, |X|_g \leq 1\right\} < \infty,
\end{align*}
which proves the claim.

\end{proof}

%%%%%%%%%%%%%%%%%%%%%%%%%%%%%%%%%%%%%%%%%%%%%%%
%%%%%%%%%%%%%%%%%%%%%%%%%%%%%%%%%%%%%%%%%%%%%%% 
% Das wird im Wesentlichen eine Kurzversion von Kapitel 4 aus der Diplomarbeit.
 
 \section{Application to scalar conservation laws}
 \label{sec:ibvp}
 We will now use the previous results to show existence, uniqueness
 and total variation estimates for an entropy solution of problem (\ref{11}), (\ref{12}) with admissible boundary conditions.
 We will proceed as in \cite{BardosLeroux} and emphasize only the boundary terms. For the rest we refer to \cite{MuellerLengeler} in which a generalization of 
 problem (\ref{11}), (\ref{12}) to Riemannian manifolds (without boundary) is treated.
 
Let $\bar M$, $u$ and $f$ be defined as in the introduction. We will write $u$ instead of $u(x,t)$ and $f(u)$ instead of $f(u(x,t),x,t)$ 
whenever this should not lead to confusion.

 Considering the characteristics of (\ref{11}), (\ref{12}) one can see that, in general, it is not possible to require $u=0$ on the whole boundary.\\
To define admissible boundary conditions and to prove existence and uniqueness of a solution of (\ref{11}), (\ref{12}) with appropriate boundary conditions
we will use the vanishing viscosity method which consists in passing to the limit, as $\epsilon >0$ tends to zero, 
in the solution $u^\epsilon$ of the parabolic regularization
\begin{align}
\partial_t u^\epsilon +\dv_g f(u^\epsilon,x,t)- \epsilon \Delta_g u^\epsilon &=0 \quad \text{ in $\bar M \times (0,T)$},\label{1}\\
u^\epsilon(\cdot,0)&= u_0^\epsilon \quad \text{in $\bar M$},\label{2}\\
u^\epsilon&=0 \quad \text{ on $\partial M \times (0,T),$}\label{3}
\end{align}\\
where $u_0^\epsilon:M\rightarrow\setR$ denotes a sequence of $u_0$ mollifying functions satisfying
\begin{align}
 \|u_0^\epsilon\|_{L^\infty(M)} &\leq \|u_0\|_{L^\infty(M)},\\
 \|u_0^\epsilon- u_0\|_{L^1(M)}&\xrightarrow {\epsilon \searrow 0} 0,\label{L1mollification} \\
 \tv(u_0^\epsilon, M)&\xrightarrow {\epsilon \searrow 0} \tv(u_0, M) ,\label{L1mollificationTV}  \\ 
 \epsilon \|u_0^\epsilon\|_{H^{2,1}(M)} &\leq c_0 \tv(u_0, M),
\end{align}
for a constant $c_0>0$.
% evaluation at time $\epsilon$ of the heat flow w.r.t the initial value $u_0$. Thus, we have $\|u_0^\epsilon\|_{L^\infty(M)}
% \leq \|u_0\|_{L^\infty(M)}$, $\|u_0^\epsilon- u_0\|_{L^1(M)}\xrightarrow {\epsilon \rightarrow 0} 0$ and as shown in \cite{Miranda} $\tv(u_0^\epsilon, M)$ converges to
% $\tv(u_0, M)$.\\
 Existence and uniqueness of a solution $u^\epsilon \in C^\infty(\bar M \times (0,T))$ of (\ref{1})-(\ref{3}) is shown in 
  \cite{Taylor}.
% \cite[p. 333, Proposition 3.1. and p. 334, Proposition 3.3]{Taylor}
 
 %Dieses Kapitel w\"urden wir auch im Wesentlichen aus der Diplomarbeit zusammenk\"urzen.
  
 \subsection{Convergence of a parabolic regularization}
 In this section we will show convergence for a subsequence of the solutions $\{u^\epsilon\}_{\epsilon >0}$ of the regularized problem \eqref{1}-\eqref{3}.
 To this end we prove boundedness of $\{u^\epsilon\}_{\epsilon >0}$ in 
 $L^{\infty}(M \times (0,T)) \cap H^{1,1}(M \times (0,T))$ uniform in $\epsilon$,
 where the product manifold $M\times(0,T)$ is endowed with the Riemannian metric $g_T$
 and apply the following theorem 
 from \cite{AubinII}.
 
%  In this section we will show the existence of a convergent subsequence $(u_{\epsilon_j})_{j\in \na}$ of $\{u^\epsilon\}_{\epsilon >0}$ which converges to an $u \in
%  L^1(M \times (0,T))$. In order to do this, we will prove the boundedness of $\{u^\epsilon\}_{\epsilon >0}$ in $H^{1,1}(M \times (0,T))$ and apply the following Theorem 
%  from \cite{Aubin}:
 
 \begin{theorem}\label{thmKondrakov}
Let $(\bar M,g)$ be a compact Riemannian manifold with Lipschitz boundary $\partial M$. Then for $p,q \geq 1$ with $1 \geq 1/q> 1/p-1/n > 0$ 
the embedding
\begin{align*}
H^{1,p} ( M) \hookrightarrow L^q( M)
\end{align*}
is compact.

\end{theorem}
\begin{proof}
See \cite[pp. 166-169, Theorem 11]{AubinII}.
\end{proof}

\begin{theorem} \label{thmH11}
%  \com{An dieser Stelle muss klar sein, aus welchen R\"aumen $u_0$ und $f$ sein sollen. Das sollten wir an den
%  Anfang dieses Kapitels stellen, denke ich. Es w\"are sinnvoll f\"ur $u_0$ nur BV anzunehmen, da diese
%  Eigenschaft erhalten bleiben und man oft mit unstetigen Anfangsdaten starten m\"ochte.
%  W\"urdest Du \"uberpr\"ufen, ob wir das machen k\"onnen, wahrscheinlich 
%  unter Zuhilfenahme einer Regularisierung von $u_0$ entsprechend \cite{MuellerLengeler}. }
 The solutions  $\{u^\epsilon\}_{\epsilon >0}$ of (\ref{1})-(\ref{3}) are
 bounded in $L^{\infty}(M \times (0,T)) \cap H^{1,1}(M \times (0,T))$ uniformly in $\epsilon$, precisely
 \begin{align}
 \|u^\epsilon\|_{L^\infty(M \times (0,T))} &\leq \|u_0^\epsilon\|_{L^\infty(M)}\leq \|u_0\|_{L^\infty(M)},\label{behLinfty}\\
 \|\partial_t u^\epsilon(\cdot, t)\|_{L^1(M)} &\leq  c_1 \tv(u_0, M) %\| u_0\| _{H^{2,1}(M)} 
 \label{partialt},\\
\| \nabla u^\epsilon (\cdot, t)\|_{L^1(M)} %&\leq \left(\| \nabla u_0^\epsilon\|_{L^1 (M)}+ c_2t\|u_0\|_{H^{2,1}(M)}\right) (1+c_3t\, e^{c_3t})\\
&\leq \left((1+c_2 t) \tv(u_0, M) + o(1)\right) (1+c_3t\, e^{c_3t})\label{Gronwallnabla}
 \end{align}
 for every $t \in (0,T)$ and consequently
 \begin{align}
 \|u^\epsilon\|_{H^{1,1}(M \times (0,T))}< c_4,\label{eq:tvestimate}  
 \end{align}
 where 
 %the constant $c_3>0$ is arbitrarily small and 
 the constants $c_1,c_2,c_3,c_4>0$ do only depend on the data $\bar M$, $g$, $T$, $f$, $\|u_0\|_{L^\infty(M)}$,
 but not on $\epsilon$.
\end{theorem}

\begin{proof}
Since there exist analogous proofs for problems without boundary conditions on manifolds, 
we will emphasize the argumentation
for the boundary terms and refer to literature for the rest.

Concerning the $L^\infty$-estimate \eqref{behLinfty} note that the proof from
\cite[pp. 131-139, Theorem 8.4]{Malek} can be transferred easily to our case, i.e. showing that $\sup_{M \times (0,T)} u_\epsilon \leq
\max\{\text{ess}\sup_M u_0,\, 0\}$ we multiply (\ref{1}) by $\Phi_\zeta'(u^\epsilon)$ where
\begin{align*}
 \Phi_\zeta(z):= \begin{cases}
\left((z-m)^2+\zeta^2\right)^{1/2}-\zeta &\text{ if } z\geq m,\\
0 &\text{ if } z < m,
\end{cases}
\end{align*}
 $m:=\max\{\text{ess}\sup_M u_0,\,0\}$ and $\zeta > 0$, integrate over $M\times (0,t)$ with $t \in (0,T)$ and let $\zeta$ tend to zero. Similarly we get
$\inf_{M \times (0,T)} u_\epsilon \geq \min\{\text{ess}\inf_M u_0,\, 0\}$.
% we introduce an approximation of the sign function $s_\eta: \re \rightarrow [-1,1]$ with $\eta >0$ and 
% \begin{align*}
% s_\eta(x):= \begin{cases}
% 1, &\text{ falls $x > \eta$}\\
% \frac{x}{\eta}, &\text{ falls $|x|\leq \eta$}\\
% -1, &\text{ falls $x < \eta$}.
% \end{cases}
% \end{align*}
% 
% Taking the derivative of equation (\ref{1}) with respect to $t$, multiplying by $s_\eta(\partial_t u^\epsilon)$ and integration over $M$, we obtain:
% \begin{align*}
%  \underbrace{\frac{d}{dt} \int_M \int_0^{\partial_t u^\epsilon} s_\eta(\xi)\, d\xi\, dv_g}_{=: I_0}&+ 
% \underbrace{\int_M \dv_g (\partial_u f(u^\epsilon,x,t)\, \partial_t u^\epsilon)  s_\eta (\partial_t u^\epsilon) \, dv_g}_{=:I_1} \\
% &+ \underbrace{\int_M \dv_g((\partial_t f)(u^\epsilon,x,t))  s_\eta (\partial_t u^\epsilon) \, dv_g}_{=: I_2=0}\\
% &= \epsilon \underbrace{\int_M \Delta_g(\partial_t u^\epsilon) \,  s_\eta (\partial_t u^\epsilon) \, dv_g}_{=:I_3}.
% \end{align*}

For the estimates of the time derivative and the total variation we proceed as in \cite{MuellerLengeler} and define a function
 $S_\eta: \re \rightarrow \re_{\geq 0}$ for $\eta > 0$ by
 \begin{align*}
S_\eta(z):= \begin{cases}
-z&\text{ if } z < -\eta,\\
\frac{z^2}{2 \eta}+\frac{\eta}{2} &\text{ if } |z| \leq \eta,\\
z &\text{ if } z> \eta.
\end{cases}
\end{align*}
The proof of \eqref{partialt} can be transferred from \cite[1022-1023]{BardosLeroux}, i.e.
taking the time derivative of \eqref{1}, multiplying with $S_\eta^\prime(\partial_tu^\epsilon)$, integrating over
$M\times(0,t)$ for $t\in(0,T)$ and letting $\eta$ tend to zero. Note, that all boundary terms vanish due to
our homogeneous boundary condition \eqref{3}.

As the proof of \eqref{Gronwallnabla} is a little more involved due to the boundary treatment, it will be presented in detail here.
From now on we will write $u$ instead of $u^\epsilon$ for better readability.
Taking the total covariant derivative of equation (\ref{1}) and using
\eqref{eq:commutator} we obtain
\begin{equation}
\partial_t \nabla u + \nabla \dv_g f(u)
= \epsilon(\Delta_g \nabla u- \ric(\nabla u, \cdot)).
\label{KovAbleitung}
\end{equation}
Multiplying (\ref{KovAbleitung}) by 
$\frac{\nabla u}{|\nabla u|_g} S_\eta'(|\nabla u|_g)$ and integration over $M$ leads to
\begin{align}
&\frac{d}{dt} \int_M  S_\eta(|\nabla u|_g)\ dv_g
+ \int_M\langle \nabla \, \dv_g f(u), \nabla u\rangle_g \frac{S_\eta'(|\nabla u|_g)}{|\nabla u|_g}\,
dv_g \notag \\
=&\epsilon\, \int_M \frac{S_\eta'(|\nabla u|_g)}{|\nabla u|_g}(\langle \Delta_g \nabla u, \nabla u\rangle_g
-\ric(\nabla u, \nabla u))\, dv_g. \label{skalar}
\end{align}
From the lines of \cite[pp. 1721-1723, Proposition 5.3]{MuellerLengeler} we know that
\begin{align*}
   \langle \nabla \, \dv_g f(u), \nabla u\rangle_g \frac{S_\eta'(|\nabla u|_g)}{|\nabla u|_g}\,
 \notag =
\left\langle \left( \nabla_{\frac{\grad_g \, u}{|\grad_g u|_g}} \partial_u f\right)(u), \, \grad_g\, u \right\rangle_g\, 
S_\eta'(|\nabla u|_g)\\
+ \dv_g\left (\partial_u f(u)\, S_\eta(|\nabla u|_g)\right)
+ \dv_g(\partial_uf(u))\,\left (|\nabla u|_g \, S_\eta ' (|\nabla u|_g) - S_\eta(|\nabla u|_g)\right ).
\end{align*}
% Considering the second term of the right-hand side we compute
% \begin{align*}
%  (\dvf(u))_{;i}=(f^j(u))_{;ji}= \partial_u f^j_{\, ,i}(u)u_{;j}+f^j_{\, ;ji}(u)+ (\partial uf^j(u)u_{;i})_{;j}.
% \end{align*}
% 
% For the first term on the right-hand side one has
% \begin{align*}
%   \partial_u f^j_{\, ,i}(u)u_{;j}u_;^{\, i}\frac{\beta'_\delta(|\nabla u|)}{|\nabla u|}=\dv(\partial_u f(u) \beta_\delta(|\nabla u|))+ \dv(\partial_u f(u))
%   (|\nabla u|) \beta'_\delta(|\nabla u|)- \beta_\delta(|\nabla u|))
% \end{align*}
% Thus (\ref{skalar}) is equivalent to
% \begin{equation}
%  \left\{\label{eq1}
% \begin{align*}
%  &\frac{d}{dt}\int_M S_\eta(|\nabla u|_g)\, dv_g
% + \int_M  \left\langle \left( \nabla_{\frac{\grad_g \, u}{|\grad_g u|_g}} \partial_u f\right)(u), \, \grad_g\, u \right\rangle_g\, 
% S_\eta'(|\nabla u|_g)\\
% &+ \dv_g\left (\partial_u f(u)\, S_\eta(|\nabla u|_g)\right)\\
% &+ \dv_g(\partial_uf(u))\,\left (|\nabla u|_g \, S_\eta ' (|\nabla u|_g) - S_\eta(|\nabla u|_g)\right )\, dv_g\\
% =&\epsilon \int_M \frac{S_\eta'(|\nabla u|_g)}{|\nabla u|_g}( \langle \Delta_g\, \nabla u, \, \nabla u\rangle_g
% - \ric(\nabla u, \, \nabla u)) dv_g. 
% \end{align*}
% \right.
% \end{equation}
% 
Concerning the first term on the right-hand side of \eqref{skalar} integration by parts yields
\begin{equation}\label{eq2}
 \begin{aligned}
\int_M  \frac{S_\eta'(|\nabla u|_g)}{|\nabla u|_g} \langle \Delta_g\, \nabla u, \, \nabla u\rangle_g \, dv_g
=& - \int_M \left\langle \nabla^2u, \, \nabla \left(\frac{S_\eta'(|\nabla u|_g)}{|\nabla u|_g} \nabla u\right)\right\rangle_g \, dv_g\\
&+\int_{\partial M} \langle \nabla u \otimes N, \nabla^2 u\rangle_g \frac{S_\eta'(|\nabla u|_g)}{|\nabla u|_g}dv_{\widetilde g},%\label{pI}
\end{aligned}
\end{equation}
where in local coordinates $ \langle \nabla u \otimes N, \nabla^2 u\rangle_g = 
 u_{;}^{\, i}\,  u_{;ij}\, N^j.$
From \cite[pp. 1721-1723, Proposition 5.3]{MuellerLengeler} we obtain positivity for the integrand of the
first term of the right-hand side of (\ref{eq2}).
% \begin{align*}
%  \left\langle \nabla^2u, \, \nabla \left(\frac{\S_\eta'(|\nabla u|_g)}{|\nabla u|_g} \nabla u\right)\right\rangle_g \geq 0.
% \end{align*}
Applying the divergence theorem on $\int_M \dv_g\left (\partial_u f(u)\, S_\eta(|\nabla u|_g)\right)\, dv_g$ and letting $\eta$ tend to zero in \eqref{skalar} we obtain
\begin{equation}\label{etazero}
\begin{aligned}
\frac{d}{dt} \int_M |\nabla u|_g\, dv_g
+\int_M \left\langle \left (\nabla_{\frac{\grad_g u}{|\grad_g u|_g}} \partial_u f\right )(u)
, \grad_g u\right\rangle_g\, dv_g\qquad\quad\\
\leq - \epsilon \int_M \ric (\nabla u , \frac{\nabla u}{|\nabla u|_g}) \, dv_g + \liminf_{\eta\searrow0}I^\eta.
\end{aligned}
\end{equation}
with
\begin{equation*}
 I^\eta:=\int_{\partial M} \epsilon\,  \left\langle \nabla u \otimes N, \nabla^2 u\right\rangle_g \frac{S_\eta'(|\nabla u|_g)}{|\nabla u|_g}- 
\left\langle \partial_u f(0),N\right\rangle_g 
\, S_\eta(|\nabla u|_g) \, dv_{\widetilde g}.
\end{equation*}
% Applying the divergence theorem we obtain that (\ref{skalar}) is equivalent to
% \begin{equation}\label{I1bisI4}
% \begin{aligned}
% &\underbrace{\frac{d}{dt}\int_M S_\eta(|\nabla u|_g)\, dv_g}_{=:I_1^\eta}\\
% &+ \underbrace{\int_M  \left\langle \left ( \nabla_{\frac{\grad_g u}{|\grad_g u|_g}} \partial_u f\right )(u), \, \grad_g u \right\rangle_g
% \, S_\eta'(|\nabla u|_g)}_{=:I_2^\eta (1)}\\
% &+ \underbrace{\dv_g(\partial_uf(u))\,\left(|\nabla u|_g \, S_\eta ' (|\nabla u|_g) - S_\eta(|\nabla u|_g)\right)\, dv_g}_{=:I_2^\eta(2)}\\
% \leq& \underbrace{- \epsilon \int_M  \ric(\nabla u, \, \nabla u) \frac{S_\eta'(|\nabla u|_g)}{|\nabla u|_g} \, dv_g}_{=:I_3^\eta}\\
% &+ \underbrace{\int_{\partial M} \epsilon\,  \left\langle \nabla u \otimes N, \nabla^2 u\right\rangle_g \frac{S_\eta'(|\nabla u|_g)}{|\nabla u|_g}- 
% \left\langle \partial_u f(0),N\right\rangle_g 
% \, S_\eta(|\nabla u|_g) \, dv_{\widetilde g}}_{=:I^\eta}. 
% \end{aligned}
% \end{equation}
% For $\eta \rightarrow 0$ we obtain
% \begin{align*}
% I_1^\eta &\xrightarrow{\eta \rightarrow 0} \frac{d}{dt} \int_M |\nabla u|_g\, dv_g,\\
% I_2^\eta &\xrightarrow{\eta\rightarrow 0} \int_M \left\langle \left (\nabla_{\frac{\grad_g u}{|\grad_g u|_g}} \partial_u f\right )(u)
% \, \grad_g u\right\rangle_g\, dv_g\\
% \text{and}\\
% I_3^\eta &\xrightarrow{\eta \rightarrow 0} - \epsilon \int_M \ric (\nabla u , \frac{\nabla u}{|\nabla u|_g}) \, dv_g.
% \end{align*}
In order to study the limit $\eta \searrow 0$ for $I^\eta$ we do some transformations first.
By the fact that $\grad_gu =N(u)N$ on $\partial M$ and consequently $|\nabla u|_g =|N(u)|$, 
 as $u\equiv0$ on $\partial M$, we have on $\partial M$
\begin{align*}
  \left\langle \nabla u \otimes N, \nabla^2 u\right\rangle_g =
  N(u) \,  \langle  N, \nabla_N \, \grad_g u\rangle_g.
\end{align*}
Considering the regularized conservation law (\ref{1}) on $\partial M$, we obtain 
% \begin{align*}
% \underbrace{\partial_t u}_{=0}+ \, \dv_g f(u)= \epsilon \Delta_g u\\
% \Leftrightarrow \langle \partial_u f(0), \grad_g u \rangle_g = \epsilon \Delta_g u 
% \end{align*}
% and with $\grad_gu=N(u)N$:
\begin{align*}
\langle \partial_u f(0), \, N\rangle_g \, N(u) = \epsilon\, \Delta_g u.
\end{align*}
Thus,
\begin{align*}
 I^\eta= \epsilon\,  \int_{\partial M}  S_\eta'(|N(u)|_g)\,  \frac{N(u)}{|N(u)|_g}\,   \langle  N, \nabla_N \, \grad_g u\rangle_g
-\Delta_g u \, \frac{S_\eta(|N( u)|_g)}{N(u)} \, dv_{\widetilde g}.
\end{align*}
For $\eta \searrow 0$ we obtain
\begin{align}
 \liminf_{\eta \searrow 0}|I^\eta|
 %\epsilon\,  \int_{\partial M} \frac{N(u)}{|N(u)|_g}\,  \langle  N, \nabla_N \, \grad_g u\rangle_g 
 %-\Delta_g u \, \frac{|N( u)|_g}{N(u)} \, dv_{\widetilde g}\\
\leq \epsilon \, \int_{\partial M} \left|\langle  N, \nabla_N \,\grad_g u\rangle_g -\Delta_g u\right|_g  \, dv_{\widetilde g}. \label{absr}
\end{align}
In geodesic boundary coordinates centered in $x \in \partial M$ we have
\begin{align*}
 \Delta_gu(x)= \sum_{i=1}^{n}
 %\partial_i 
 \partial_i^2 u(x)
 = \partial_n^2u(x)
\end{align*}
regarding the fact that $u \equiv 0$ on $\partial M$.
% \begin{align*}
%  \Delta_g u(p)= \partial_n(\partial_nu).
% \end{align*}
Extending $N$ onto a small neighborhood of $\partial M$ by
$N=-\partial_n$
we obtain at $x \in \partial M$
\begin{align*}
\Delta_g u= N (N(u))= N( \langle N, \grad_g u \rangle_g)
= \langle \nabla_N N,\grad_g u \rangle_g + \langle N, \nabla_N \grad_g u \rangle_g.
\end{align*}
Note that the above expression is independent of the choice of coordinates.
Thus,
\begin{align*}
 \liminf_{\eta\searrow0}|I^\eta|
%\leq&\epsilon\, \int_{\partial M}| \langle \nabla_N N, \grad_g u \rangle_g|_g \, dv_{\widetilde g}\notag \\
%\leq& \epsilon \, \int_{\partial M} |\nabla_N N|_g\,  |\nabla u|_g\, dv_{\widetilde g}\\
\leq  c\, \epsilon \,  \int_{\partial M}| \nabla u|_g \, dv_{\widetilde g}
\end{align*}
with $c:=\|\nabla_N N\|_{L^\infty(\partial M)} <\infty$ since $N$ is smooth and $\partial M$ compact.
A repetition of the proof for the Euclidean case \cite[p. 92, Lemma A.3]{BardosBrezis} yields the 
analogous result,
\begin{align*}
 \int_{\partial M} |\nabla u|_g \, dv_{\widetilde g} \leq \int_M |\Delta_g u|\, dv_g,
\end{align*}
on manifolds. Using again equation (\ref{1}), we obtain
\begin{align*}
  c\, \epsilon \,  \int_{\partial M}| \nabla u|_g \, dv_{\widetilde g} 
  \leq c' \left( \|\partial_t u\|_{L^1(M)}+ \|\nabla u\|_{L^1(M)}\right) 
\end{align*}
 for a constant $c'=c'(\bar M,g,T,f,\|u_0\|_{L^\infty(M)})>0$.
% Summarizing the previous results we obtain from \eqref{etazero}
%  \begin{align*}
% &\frac{d}{dt} \int_M |\nabla u|_g\, dv_g+ \int_M \left \langle \left(\nabla_{\frac{\grad_g u}{| \grad_g u|}} \partial_u f\right)(u), \, \grad_g u\right\rangle_g
% \, dv_g\\
% &+ c_1( \|\partial_t u\|_{L^1(M)}+ \|\nabla u\|_{L^1(M)})\\
% \leq& - \epsilon \int_M \frac{1}{|\nabla u|_g} \ric (\nabla u , \nabla u) \, dv_g.
% \end{align*}
Since $\ric(\nabla u, \nabla u) \leq c'' |\ric|_g |\nabla u|_g^2$ with a constant $c''=c''(\bar M,g,T)>0$, 
we obtain from \eqref{etazero}
\begin{align*}
&\frac{d}{dt} \int_M |\nabla u|_g\, dv_g 
\leq c_3 \|\nabla u\|_{L^1(M)}+ c' \|\partial_t u\|_{L^1(M)}\\
\end{align*}
with
\begin{equation*}
 c_3:=\sup_{\bar u, X}\|\langle\nabla_{X} \partial_u f(\bar u),X\rangle_g\|_{L^\infty(M\times(0,T))}+  \epsilon \|\ric\|_{L^\infty(M)}+c',
\end{equation*}
where the supremum is taken over all real numbers $|\bar u|\leq\|u_0\|_{L^\infty(M)}$
and all smooth vector fields $X$ with $|X|_g\leq1$.
Integration over $(0,t)$ together with (\ref{partialt}) and \eqref{L1mollificationTV} yields for almost every $t \in (0,T)$
\begin{align*}
 \int_M |\nabla u(\cdot, t)|_g \, dv_g 
&\leq \| \nabla u^\epsilon_0\|_{L^1 (M)}+ c_3 \int_0^t \| \nabla u(\cdot,\tau)\|_{L^1(M)} \, d\tau+ tc'c_1 \tv(u_0,M)\\
&\leq (1+c_2 t) \tv(u_0, M) + o(1) + c_3 \int_0^t \| \nabla u(\cdot,\tau)\|_{L^1(M)} \, d\tau
\end{align*}
with $c_2:=c'c_1$ 
%depending on $\bar M,g,T,f,u_0$, but not on $\epsilon$ and $c_3>0$ arbitrarily small.
and $o(1)\rightarrow 0$ for $\epsilon\searrow0$.
Finally, Gronwall's Inequality completes the proof of \eqref{Gronwallnabla} and consequently of  \eqref{eq:tvestimate}.
% \begin{align}
% \| \nabla u^\epsilon (\cdot, t)\|_{L^1(M)} \leq \left(\| \nabla u_0\|_{L^1 (M)}+ c_4t\|u_0\|_{H^{2,1}(M)}\right) (1+c_3t\, e^{c_3t})\label{Gronwallnabla}
% \end{align}
% for almost every $t \in (0,T)$ with a constant $c_3=c_3(\bar M,g,T,f,u_0)>0$.
% \begin{align*}
%  \|u^\epsilon\|_{H^{1,1}(M \times (0,T))}=\|u^\epsilon\|_{L^1(M\times (0,T))}+ \|\nabla^{g_T}u^\epsilon\|_{L^1(M\times (0,T))}
% \end{align*}
% together with
% \begin{align*}
%  \|u\|_{H^{1,1}(M \times (0,T))}\leq c_5 
% \end{align*}
% yields \eqref{eq:tvestimate} with
% \begin{align*}
%  c_5=\|u\|_{L^1(M\times (0,T))}+ \|\nabla u\|_{L^1(M\times (0,T))}+ \|\partial_t u\|_{L^1(M\times (0,T))}. 
% \end{align*}
% .

\end{proof}

% \begin{corollary}\label{corsequentially}
%  The family $\{u^\epsilon\}_{\epsilon >0}$ of solutions of problem (\ref{1}) - (\ref{3}) is sequentially compact in $L^1(M\times (0,T))$.
% \end{corollary}
% \com{Das Corollary \ref{corsequentially} streichen!}
% \begin{proof}
% The fact $u\in H^{1,1}(M\times (0,T))$ follows from Theorem \ref{thmH11}.\\
% Theorem \ref{thmKondrakov} implies that the embedding
% \begin{align*}
%  H^{1,1}(M\times (0,T)) \hookrightarrow L^1(M\times (0,T))
% \end{align*}
% is compact and thus $\{u^\epsilon\}_{\epsilon >0}$ is sequentially compact in $L^1(M\times (0,T))$.
% \end{proof}
An application of Theorem \ref{thmKondrakov} leads to the following corollary of Theorem \ref{thmH11}.
\begin{corollary}[Viscosity limit]
 There is a subsequence
$(u^{\epsilon_j})_{j \in \mathbb N}$ of $\{u^\epsilon\}_{\epsilon > 0}$ and a function $u \in L^1(M\times (0,T))$
such that $\|u^{\epsilon_j} - u \|_{L^1(M\times (0,T))} \rightarrow 0$ for $j\rightarrow \infty$.
Such a limit function $u$ is called viscosity limit of (\ref{1})-(\ref{3}).
\end{corollary}

% Note, that as a special case of Theorem \ref{thmKondrakov} the embedding
% \begin{align*}
% H^{1,1}(M\times (0,T)) \hookrightarrow L^1(M\times (0,T)) 
% \end{align*}
% is compact. Consequently, \eqref{eq:tvestimate} implies the existence of a subsequence
% $(u^{\epsilon_j})_{j \in \mathbb N}$ and a function $u \in L^1(M\times (0,T))$, in the following denoted
% as a viscosity limit of (\ref{1})-(\ref{3}),
% such that $\|u^{\epsilon_j} - u \|_{L^1(M\times (0,T))} \rightarrow 0$ for $j\rightarrow \infty$.

% \begin{remark}
% Since Theorem \ref{thmH11} proves $u^\epsilon \in H^{1,1}(M\times (0,T))$, Theorem \ref{thmKondrakov} implies that the embedding 
% \begin{align*}
%   H^{1,1}(M\times (0,T)) \hookrightarrow L^1(M\times (0,T))
%  \end{align*}
%  is compact and thus $\{u^\epsilon\}_{\epsilon >0}$ is sequentially compact in $L^1(M\times (0,T))$. We will denote the limit function $u \in L^1(M\times (0,T))$
%  of any convergent subsequence $(u^{\epsilon}_j)_{j \in \mathbb N}$ as viscosity limit of (\ref{1})-(\ref{3}).

% We call this limit function $u$ the viscosity limit of (\ref{1}) - (\ref{3}).
% \com{Hier den Begriff Viskositaetslimes so einzufuhren ist problematisch, weil wir
% an dieser Stelle keine Eindeutigkeit f\"ur diesen Begriff haben!
% Wir k\"onnten sagen, wir mit Viskositaetslimes einen entsprechenden Grenzwert bezeichnen.}
%\end{remark}

 \subsection{Existence of an entropy solution}\label{sec:existence}
First, we motivate a formulation of boundary conditions analogously to \cite[126-127]{BardosLeroux}.
For a moment, assume the flux function
$f=f(u,x,t)$ to be monotone in $u$.
In this case, outflow boundary points $x\in\partial M$ at time $t\in(0,T)$ are characterized by the property
\begin{align}
\left \langle \frac{f(Tu)-f(k)}{Tu-k}, N \right\rangle_g > 0\quad \forall k \in \re \label{outflow}
\end{align}
where $Tu$ denotes the trace of $u$ on $\partial M$ and inflow boundaries are characterized conversely.
% with
% \begin{align*}
% \left \langle \frac{f(Tu,x,t)-f(k,x,t)}{Tu-k}, N \right\rangle_g\begin{cases} >0 \text{ for leaving data}\\
% <0 \text{ for entering data}.
% \end{cases}
% \end{align*}
We want to have a  boundary condition which 
\begin{enumerate}
\item requires $u=0$ on $\partial M$ if the data are entering $\bar M$, which means that $u$ is not determined by the initial data or some other boundary data and
\item is a trivial condition if the data are leaving $\bar M$.
\end{enumerate}
This is ensured by the following boundary condition:
\begin{align}
\min_{k \in I(Tu,0)} \{\sgn(Tu)\, \langle f(Tu)- f(k), N\rangle_g \}=0 \label{boundarycondition}
\end{align}
almost everywhere on $\partial M\times (0,T)$ with
$I(Tu, 0):= [ \min\{Tu,0\}, \max\{Tu,0\}]$.

Although these explanations only work for monotone functions $f$, we will see that (\ref{boundarycondition}) is a valid  boundary condition for
problem (\ref{11}), (\ref{12}), even if $f$ is not monotone.\\

We now give the definition for an entropy solution of problem (\ref{11}), (\ref{12}), (\ref{boundarycondition}).
\begin{definition}[Entropy solution] \label{defiLoesungsbegriff}
We call a function $u \in \BV(M \times (0,T))\cap L^\infty(M\times (0,T))$ entropy solution of problem 
\eqref{11}, \eqref{12}, \eqref{boundarycondition}
%with boundary condition
% \begin{align}
% \min_{k \in I(Tu,k)} \{\sgn(Tu)\, \langle f(Tu,x,t)- f(k,x,t), N\rangle_g \}=0 \text{ on $\partial M \times (0,T)$} \label{boundarycondition}
% \end{align}
if for every $k \in \re$ and every $\phi \in C_0^\infty (\bar M \times (0,T))$, 
$ \phi(x,t)\geq 0$ the inequality
\begin{equation}
\label{entrol}
\begin{aligned}
\int_M \int_0^T  |u -k| \, \partial_t \phi
 +\sgn (u-k)\,  \langle f(u)- f(k), \grad_g  \phi \rangle_g    \, dt\, dv_g\\
 +\int_{\partial M} \int_0^T \sgn(k) \,  \langle f(Tu)-f(k) , N\rangle_g \, \phi \, dt \, dv_{\widetilde g} \geq 0 
\end{aligned}
\end{equation}
 holds and the initial condition  (\ref{12}) is fulfilled by the trace $Tu|_{t=0}$ almost everywhere in $M$.
\end{definition}

\begin{remark}
Setting $k=\sup_{M \times (0,T)} |u(x,t)|$ and $k=\inf_{M \times (0,T)} |u(x,t)|$ in (\ref{entrol}) it is easy to prove that an entropy solution 
is a weak solution of (\ref{11}).
The fact that \eqref{entrol} implies the boundary condition \eqref{boundarycondition}
is proved by setting $\phi= R_\delta\bar \phi$ in \eqref{entrol} with 
$\bar \phi \in C_0^\infty(\bar M \times (0,T))$, $\bar \phi \geq 0$ and $R_\delta$ as in (\ref{rhodelta}) and letting $\delta$ tend to zero.
% 
% $u$  is entropy solution  $\Rightarrow$  $u$  fulfills the boundary condition (\ref{boundarycondition}) for almost every $(x,t) \in \partial M\times (0,T)$,\\
% we set in (\ref{entrol})  $\phi= R_\delta\bar \phi$ with $\bar \phi \in C_0^\infty(\bar M \times (0,T)$, $\bar \phi \geq 0$ and $R_\delta$ as in the proof of
% Theorem (\ref{SpursatzM}). By $\delta$ tending to zero, we obtain 
% \begin{align}
%  \int_{\partial M}  \int_0^T \left(\sgn (Tu-k)+\sgn(k)\right) \,  \langle f(Tu, x,t)- f(k,x,t), N \rangle_g \, \bar\phi\, dt\, dv_{\widetilde g} \geq 0 
%  \label{AbschaetzungRandbed}
% \end{align}
% for every $\bar \phi \in C_0^\infty(\bar M \times (0,T))$, $\bar \phi \geq 0$, which implies
% \begin{align*}
%  \int_{\partial M}  \int_0^T \sgn (Tu) \,  \langle f(Tu, x,t)- f(k,x,t), N \rangle \, \bar\phi\, dt\, dv_{\widetilde g} \geq 0
% \end{align*}
% for every  $\bar \phi \in C_0^\infty(\bar M\times (0,T))$, $\bar \phi \geq 0$.\\
% Thus, we have
% \begin{align*}
%  \sgn (Tu) \,  \langle f(Tu, x,t)- f(k,x,t), N \rangle \geq 0
% \end{align*}
% for almost every $(x,t) \in \partial M \times (0,T)$, which proves that $u$ fulfills the boundary condition (\ref{boundarycondition}).
\end{remark}

\begin{theorem}[Existence of an entropy solution]\label{thmexistence}
Any viscosity limit $u$ of (\ref{1})-(\ref{3})  is an entropy solution of problem (\ref{11}), (\ref{12}), (\ref{boundarycondition}).
% \com{Ich f\"ande es besser, wenn hier noch etwas mehr zu den Eigenschaften von $u$ stehen w\"urde.
% Durch den Grenz\"ubergang \"ubertragen sich ja viele Eigenschaften, die momentan in Corollary \ref{corAnfangsbed}
% stehen. Ich glaube, ich f\"ande es besser, wenn wir die ganzen Sachen zum Grenz\"ubergang (also ab Remark 4.4) in diesen Abschnitt
% packen wuerden, damit wir uns hier dann nicht wiederholen.}
\end{theorem}

 To prove Theorem \ref{thmexistence}, we will use the following Corollary \ref{corAnfangsbed} and Lemma \ref{lemSpurFluss}.
 
 \begin{corollary} \label{corAnfangsbed}
 Any viscosity limit $u$ of (\ref{1})-(\ref{3})  belongs to $\BV(M\times (0,T)) \cap L^\infty(M \times (0,T))$ and has a trace for $t=0$ satisfying the initial 
 condition (\ref{12}) almost everywhere on $M$.
\end{corollary}

\begin{proof}
 Since (\ref{behLinfty}), (\ref{partialt}) and \eqref{Gronwallnabla} hold and since $\tv$ is lower semicontinuous w.r.t $L^1$-convergence we obtain 
 \begin{align*}
  u \in \BV(M \times (0,T)) \cap L^\infty(M \times (0,T)).
 \end{align*}
Let $Tu|_{t=0}$ be the trace of $u$ for $t=0$ whose existence is ensured by Theorem \ref{SpursatzM}.
Then we have for $(u^{\epsilon_j})_{j\in \na}$, a converging subsequence of $\{u^\epsilon\}_{\epsilon>0}$,
\begin{align*}
\|Tu|_{t=0}- u_0\|_{L^1(M)}
\leq&\| Tu|_{t=0} (\cdot) - u(\cdot, t)\|_{L^1(M)}+ 
\|u(\cdot, t)- u^{\epsilon_j}(\cdot, t)||_{L^1(M)}\\
&+
\| u^{\epsilon_j}(\cdot,t)-u_0^{\epsilon_j}(\cdot)\|_{L^1(M)}+
\|u_0^{\epsilon_j}-u_0\|_{L^1(M)}.
\end{align*}
Since (\ref{partialt}) holds inedependently of $\epsilon_j$, the third term of the right hand side is bounded by $tc$ with a constant $c>0$. Thus, letting first $j$ tend to $\infty$ and then $t$ to zero, we obtain
\begin{align*}
 \lim_{t \searrow 0} \|Tu|_{t=0}- u_0\|_{L^1(M)}=0,
\end{align*}
analogously, to the Euclidean case in \cite[pp. 125-126]{BardosLeroux}.
\end{proof}

\begin{lemma}\label{lemSpurFluss}
Let $u \in \BV(M)\cap L^\infty(M)$, 
%\rem{and  $f(u,\cdot)$ for every $u \in \re$ be a smooth vector field on $M$ depending smooth on $u$}
%\com{Ich w\"urde vorschlagen hier dasselbe $f$ zu nehmen, das wir sonst auch haben und 
%es am Anfang des Kapitels einzuf\"uhren.}
 $R_\delta$ as in (\ref{rhodelta}) and  $\phi \in C_0^\infty(\bar M)$. Then we have
\begin{align}
 \lim_{\delta \searrow 0}\int_M \langle f(u), \grad_g R_\delta\rangle_g \, \phi\, dv_g =  \int_{\partial M} 
 \langle f(Tu), N\rangle_g\, \phi\, dv_{\widetilde g}.\label{spursL1}
\end{align}

\end{lemma}
\begin{proof}
Without loss of generality we neglect the $t$-dependence of $f$ in the proof.
For $X \in \Gamma(TM)$ and $v\in \BV(M)$ we obtain by Corollary \ref{kordirFolgerungSpur}
\begin{align}
\lim_{\delta \searrow 0} \int_M v \langle \grad_g R_\delta, X\rangle_g\, dv_g
 = \int_{\partial M} Tv \langle X, N\rangle_g \,  dv_{\widetilde g} \label{dirFolgerungSpursatz}.
\end{align}
Let $\{(U_i, \varphi_i)\}_{i \in I}$ be a finite collection of charts which covers $M$ 
and $\{\psi_i\}_{i \in I}$ a partition of unity subordinate to this
cover. For $i \in I$ and $1 \leq l \leq n$ we define
\begin{align*}
 \widetilde f^l_i (x):= \begin{cases}
                         f^l_i(u(x),x) &\text{ if $x \in U_i$,}\\
                         0 &\text{ otherwise}
                        \end{cases}
\end{align*}
where $f_i^l(u(x),x)$ denotes the $l$-th component of $f(u(x),x)$ on $U_i$ 
relating to $\varphi_i$. 
Note that $f_i^l(u,\cdot)\in BV(U_i)$ due to Lemma \ref{korfBVM} and hence an application of Theorem \ref{SpursatzM} yields
\begin{align*}
\tv( \widetilde f^l_i, M)
 &=\sup\left\{-\int_{U_i} \langle X, \sigma_{f^l_i(u, \cdot)}\rangle_g \, d |D f^l_i(u, \cdot)|\right.\\
 &\qquad\quad\ \ \left.+ \int_{\partial U_i}  Tf^l_i(u, \cdot)\langle X, N\rangle_g\, dv_{\widetilde g}
 |\, X\in \Gamma(TM), \, |X|_g \leq 1\right\}\\
 &\leq \tv(f^l_i(u, \cdot), U_i)+ \|Tf^l_i(u, \cdot)\|_{L^\infty(\partial U_i)} \text{vol}_{\widetilde g}(\partial U_i)\\
 &< \infty
\end{align*}
and thus, (\ref{dirFolgerungSpursatz}) is true for $v= \widetilde f^l_i$. Setting $X= \phi\psi_i \partial_l$, where $\partial_l$ is defined by $\varphi_i$, we obtain
\begin{align*}
\int_{U_i}  f^l_i (u, \cdot) \, {R_\delta}_{;}^{\, j}\,  \phi \psi_i \, g_{jl}\, dv_g
&\xrightarrow{\delta \searrow 0} \int_{U_i \cap \partial M} T f^l_i(u, \cdot) \, \phi \psi_i\,  N^j g_{jl} \,  dv_{\widetilde g}.
\end{align*}
% where we do not take the sum over $l$.\\
After summation over $i\in I$ it remains to show that we can 
replace $Tf_i^l(u,\cdot)$ by $f_i^l(Tu,\cdot)$.
% To this end we have to extend (\ref{CorTraceComm}) in Corollary \ref{korEigenschaftenTu}
% to the case that the function $f_i^l$ has two arguments, $u(x)$ and $x$.\\
To this end, let $x_0 \in U_i\cap\partial M$ be fixed, meaning $x_0$ is not the variable of integration,
then due to Lemma \ref{korEigenschaftenTu} we have for $\vol_{\widetilde g}$-almost every such $x_0$
\begin{align*}
 &\left| T[f_i^l(u(x_0), x_0)]-  f_i^l(Tu(x_0),x_0) \right| \\
 =&\left|\lim_{\rho\searrow 0} \dashint_{B_\rho^g(x_0)} f_i^l(u(\cdot), \cdot )\, dv_g- \dashint_{B_r^g(x_0)} f_i^l(Tu(x_0),x_0)\, dv_g\right|\\
 \leq& \liminf_{ \rho \searrow 0}  \dashint_{B_\rho^g(x_0)} | f_i^l(u(\cdot), \cdot)-f_i^l(u(\cdot) , x_0)|\, dv_g\\
 &+\liminf_{ \rho \searrow 0} \dashint_{B_\rho^g(x_0)}|f_i^l(u(\cdot), x_0)-f_i^l(Tu(x_0),x_0)|
 \,dv_g=0
% \leq& \lim_{ r \rightarrow 0} \left[ \underbrace{ \dashint_{B_r(p)} | f_i^l(u(\cdot), \cdot)-f_i^l(u(\cdot), p)|\, dv_g}_{=:A_1^r}
%  + \underbrace{\dashint_{B_r(p)}|f_i^l(u(\cdot), p)-f_i^l(Tu(p),p)|
%  \,dv_g}_{=:A_2^r} \right],
 \end{align*}
% where $ \lim_{r \rightarrow 0} A_1^r=0$ because of Lemma \ref{lemLebesgueM}, and $ \lim_{r \rightarrow 0} A_2^r=0$ 
because of Corollary \ref{korEigenschaftenTu} (\ref{CorTraceComm}) and the regularity of $u$ and $f$.\\
% Thus, almost everywhere on $\partial M$ we have
% \begin{align*}
%  \int_{U_i} f^l_i(u, \cdot) \, {R_\delta}_{;}^{\, j}\,  \phi \psi_i \, g_{jl}\, dv_g
% \xrightarrow{\delta \rightarrow 0} \int_{U_i \cap \partial M}  f^l_i(Tu,\cdot) \, \phi \psi_i\,  N^j g_{jl} \,  dv_{\widetilde g}
% \end{align*}
% Taking the sum first over $l$, then over $i$, we obtain
% 
% \begin{align*}
%  \int_{M} \langle f(u, \cdot),\, \grad_g R_\delta\rangle_g \,  \phi \, dv_g
% \xrightarrow{\delta \rightarrow 0} \int_{ \partial M}  \langle f(Tu,\cdot),N\rangle_g \, \phi \,  dv_{\widetilde g},
% \end{align*}
% 
% which proves the claim.
\end{proof}

\begin{proof}[Proof of Theorem \ref{thmexistence}]
We define an approximation $s_\eta: \re \rightarrow [-1,1]$ of the sign function with
\begin{align*}
s_\eta(z):=S_\eta^\prime(z)= \begin{cases}
-1 &\text{ if $z < -\eta$}\\
\frac{z}{\eta} &\text{ if $|z|\leq \eta$},\\
1 &\text{ if $z > \eta$}.
\end{cases}
\end{align*}
Multiplying (\ref{1}) by $s_\eta(u^\epsilon -k)\phi$,
where $k \in \re$, $\phi \in C_0^\infty(\bar M\times (0,T))$, $\phi \geq 0$
and integration over $M\times (0,T)$ yields
% \begin{equation}
% \left\{ 
% \begin{align}
% &\int_M \int_0^T \partial_t u^\epsilon\,  s_\eta(u^\epsilon-k) \, \phi \, dt\, dv_g\\
% &+\int_M \int_0^T \dv_g f(u^\epsilon, x,t) \,  s_\eta(u^\epsilon-k) \, \phi \, dt\, dv_g \\
% &-\int_M \int_0^T\epsilon \Delta_g u^\epsilon  \,  s_\eta(u^\epsilon-k) \, \phi \, dt\, dv_g  \\
% &= 0.
% \end{align}
% \right. \label{eq1existence}
% \end{equation}
% Using integration by parts concerning $t$ for the first line of (\ref{eq1existence}), integration by parts concerning $x$ for the second and third
% line and the fact that
% \begin{align*}
% &\int_M \int_0^T \dv_g f(u^\epsilon, x,t) \,  s_\eta(u^\epsilon-k) \, \phi \, dt\, dv_g\\
% =&\int_M \int_0^T \dv_g \left [f(u^\epsilon, x,t)- f(k,x,t)\right] \,  s_\eta(u^\epsilon-k) \, \phi \, dt\, dv_g\\
% \end{align*}
% we obtain
 \begin{equation}
\begin{aligned}
  &\int_M \int_0^T \left[ \int_k^{u^\epsilon} s_\eta(z-k) \, dz\right]\, \partial_t \phi \, dt\, dv_g \\
 &+\int_M \int_0^T\langle  f(u^\epsilon)- f(k), \grad_g  u^\epsilon \rangle_g  \,  s_\eta'(u^\epsilon-k) \, \phi \, dt\, dv_g\\
&+\int_M \int_0^T \langle f(u^\epsilon)- f(k), \grad_g  \phi \rangle_g  \,  s_\eta(u^\epsilon-k)  \, dt\, dv_g\\
=&\, \epsilon \, \int_M \int_0^T \, |\grad_g u^\epsilon|_g^2   \,  s_\eta'(u^\epsilon-k) \, \phi \, dt\, dv_g\\
  &+\epsilon \, \int_M \int_0^T \langle \grad_g u^\epsilon, \grad_g \phi\rangle_g  \,  s_\eta(u^\epsilon-k) \, dt\, dv_g\\
 &+ \epsilon \,  \int_{\partial M} \int_0^T N( u^\epsilon)\, s_\eta(k) \, \phi \, dt\, dv_{\widetilde g}\\
  &- \int_{\partial M} \int_0^T \langle f(0)- f(k), N \rangle_g  \,  s_\eta(k)\, \phi  \, dt\, dv_{\widetilde g}
\end{aligned}
 \label{eqex2}
 \end{equation}
where we used integration by parts, $\dv_g f(k)=0$ and the definition of $\phi$.
Since the forth line of (\ref{eqex2}) is nonnegative and Lemma \ref{lemSaks} yields 
that the second line of (\ref{eqex2}) tends to zero for $\eta \searrow 0$,
we obtain in the limit $\eta \searrow 0$:
\begin{equation}
\begin{aligned}
&\int_M \int_0^T  |u^\epsilon -k| \, \partial_t \phi \, dt\, dv_g\\
 &+\int_M \int_0^T \langle f(u^\epsilon)- f(k), \grad_g  \phi \rangle_g  \,  \sgn (u^\epsilon-k)  \, dt\, dv_g\\
\geq& \, \epsilon \, \int_M \int_0^T \langle \grad_g u^\epsilon, \grad_g \phi\rangle_g  \, \sgn(u^\epsilon-k) \, dt\, dv_g\\
 &+ \epsilon \,  \int_{\partial M} \int_0^T N( u^\epsilon) \, \sgn(k) \, \phi \, dt\, dv_{\widetilde g}\\
 &- \int_{\partial M} \int_0^T \langle f(0)- f(k), N \rangle_g  \,  \sgn(k)\, \phi  \, dt\, dv_{\widetilde g}.
\end{aligned}
\label{eqex3}
\end{equation}
Next, we consider $\epsilon \searrow 0$. 
Since the total variation of $u^\epsilon$ on $M$ is bounded uniformly in $\epsilon$
(cf. \eqref{Gronwallnabla}) 
%and $\phi \in C_0^\infty(M\times (0,T))$,
the third line of (\ref{eqex3}) tends to zero for $\epsilon \searrow 0$.\\ 
With regard to the forth line of (\ref{eqex3}) we insert $R_\delta$, defined 
in (\ref{rhodelta}), apply the divergence theorem and use (\ref{1})
and \eqref{Gronwallnabla} in order to conclude
\begin{align*}
& \lim_{\epsilon \searrow 0} \epsilon \,  \int_{\partial M} \int_0^T N(u^\epsilon) \, \phi \, dt\, dv_{\widetilde g}\\
 =& \lim_{\epsilon \searrow 0}
\epsilon \int_M \int_0^T \Delta_g u^\epsilon \, \phi\, R_\delta+ \langle \grad_g u^\epsilon, \grad_g(\phi\, R_\delta)\rangle_g \, dt\, dv_g\\
=&\lim_{\epsilon \searrow 0}\int_M \int_0^T (\partial_t u^\epsilon + \dv_gf(u^\epsilon))
 \, \phi\, R_\delta\, dt\, dv_g\\
% &+\underbrace{\lim_{\epsilon \rightarrow 0}\epsilon \int_M \int_0^T\langle \grad_g u^\epsilon, \grad_g(\phi\, R_\delta)\rangle \, dt\, dv_g}_ {= 0}\\
 =& - \int_M \int_0^T (u \, \partial_t \phi + \langle f(u), \grad_g \phi\rangle_g ) \, R_\delta \, dt \, dv_g\\
 &- \int_M \int_0^T \langle f(u), \grad_g   R_\delta\rangle_g \phi\, dt \, dv_g \\
 &+ \int_{\partial M} \int_0^T \langle f(0), N\rangle_g \, \phi \,  dt \, dv_{\widetilde g}.
\end{align*}
With the fact that the first term on the right-hand side tends to zero for $\delta \searrow 0$ 
and with Lemma \ref{lemSpurFluss} applied to the second term on the right-hand side
we conclude that \eqref{eqex3} in the limit $\epsilon\searrow0$ implies that any viscosity limit $u$ of \eqref{1}-\eqref{3}
fulfills the entropy inequalities \eqref{entrol}.
% \begin{align*}
% \lim_{\delta \rightarrow 0}- \int_M \int_0^T \langle f(u, x,t), \grad_g   R_\delta\rangle_g \phi\, dt \, dv_g= 
% -\int_{\partial M} \int_0^T \langle f(Tu,x,t), N\rangle_g\phi \, dt\, dv_{\widetilde g}.
% \end{align*}
% Thus for the forth line of (\ref{eqex3}) we have
% \begin{align*}
% &\lim_{\epsilon \rightarrow 0}\int_M \int_0^T N(u^\epsilon)\,  \sgn(k)\, dt\, dv_g\\
% =&\int_{\partial M}\int_0^T \langle f(0,x,t)- f(Tu,x,t),
% N\rangle_g \, \phi\,\sgn(k)\,  dt\, dv_{\widetilde g}
% \end{align*}
% and $\epsilon \rightarrow 0$ in (\ref{eqex3}) yields
% \begin{align*}
%  &\int_M \int_0^T  |u -k| \, \partial_t \phi
%  + \langle f(u, x,t)- f(k,x,t), \grad_g  \phi \rangle_g  \,  \sgn (u-k)  \, dt\, dv_g\\
%  &+\int_{\partial M} \int_0^T \langle f(Tu,x,t)-f(k,x,t) , N\rangle_g  \,\sgn(k) \, \phi \, dt \, dv_{\widetilde g} \\
%  &\geq 0.
% \end{align*}
% Since $k \in \re$ and $\phi \in C_0^\infty(\bar M \times (0,T))$, $\phi \geq 0$ are arbitrary and with Corollary \ref{corAnfangsbed} the claim is proved.
\end{proof}
%%%%%%%%%%%%%%%%%%%%%%%%%%%%%%%%%%%%%%%%%%%%%%%%%%%%%
% \com{Allgemein zu diesem Abschnitt ist mir aufgefallen, dass wir in Integralen oft $f(u,x,t)$ etc. schreiben,
% in Section 4.1. jedoch nur $f(u)$. Mir pers\"onlich gef\"allt es besser, wenn wir die k\"urzre Variante
% w\"ahlen, also die Integrationsvariable nicht explizit hinschreiben, d.h. $f(u)$ finde ich besser.
% Auch ist $f(u,x,t)$ ja nicht ganz konsistent, da es eigentlich $f(u(x,t),x,t)$ heissen m\"usste,
% was defninitiv zu lang ist. Zu beachten ist jedoch die Ausnahme beim Beweis von Lemma \ref{lemSpurFluss},
% wo wir das $p$ hinschreiben m\"ussen, um klar zu machen \"uber welche Argumente wir nun integrieren.
% }
\subsection{Uniqueness of the entropy solution}
To prove uniqueness we will use Kruzkov's \cite{Kruzkov} technique of doubling the variables which was generalized by
Lengeler and M\"uller \cite{MuellerLengeler} to the case of closed Riemannian manifolds. In this section we adapt their work
to the case of compact Riemannian manifolds with boundary.

We need the following Lemma from Kruzkov \cite{Kruzkov}.
\begin{lemma} \label{lemlip}
If a function $h\in C(\re)$ satisfies a Lipschitz condition on an interval $[-z,z] \subset \re $ with constant $L>0$, then the function
$q(z_1, z_2):= \sgn(z_1-z_2)(h(z_1)-h(z_2))$ satisfies the Lipschitz condition in $z_1$ and $z_2$ with the same constant $L$.
\end{lemma}
% \begin{proof}
% With
% \begin{align*}
%  &\left |q(u_1, v_1)-q(u_2,v_2)\right|\\
%  =&\left|\int_0^1  \frac{\partial}{\partial_t}\left[ q \left( (u_2, v_2)^T + 
%  t (u_1-u_2,\, v_1-v_2)^T \right) \right]
%  \, dt \right|
% \end{align*}
% this can be easily proved.
% \end{proof}

\begin{theorem}[Uniqueness of the entropy solution]
The entropy solution of problem (\ref{11}), (\ref{12}), (\ref{boundarycondition})  is unique. 
\end{theorem}
\begin{proof}
We assume that there exist two entropy solutions $u$ and $v$. Using the doubling of variables technique of 
Kruzkov (cf. \cite{Kruzkov})
we consider (\ref{entrol}) with $\phi= \phi(x,t,y,s) \in C_0^\infty(M\times (0,T) \times M \times (0,T))$ 
first  for $u$ with $k=v(y,s)$ and integrate over
$M \times (0,T)$ w.r.t. $(y,s)$, and then for $v$ with $k=u(x,t)$ and integrate over $M \times (0,T)$ w.r.t. $(x,t)$.
Summation of the two inequalities yields
\begin{equation}
\begin{aligned}
&\int_0^T\int_M \int_0^T \int_M |u(x,t)-v(y,s)|(\partial_t \phi+ \partial_s \phi) \\
&+ \langle q(u(x,t), v(y,s),x,t),\grad_g^x \phi \rangle_g \\ 
&+ \langle q(v(y,s),u(x,t),y,s),\grad_g^y \phi \rangle_g\, 
dv_g(y) \,ds\, dv_g(x)\, dt   \geq 0
\end{aligned}
 \label {eqEind1}
\end{equation}
with
\begin{align*}
q(u,k,x,t):= \sgn(u-k)\left(f(u,x,t)-f(k,x,t)\right).
\end{align*}
We set
\begin{align*}
\phi(x,t,y,s):= \psi(t)\bar \psi(x)\omega_{\bar{\epsilon}}(t-s) \kappa_\epsilon(x,y),
\end{align*}
where  $\psi \in C_0^\infty((0,T))$ with $\psi\geq 0$ and $\bar \psi \in C_0^\infty( M)$ with $\bar\psi\geq 0$,
$\omega_{\bar{\epsilon}}(s):= \frac{1}{\bar{\epsilon}}\omega(\frac{s}{\bar{\epsilon}})$ 
with $\omega \in C^\infty(\re)$, $\text{supp}(\omega) \subset (-1,1)$, $\omega \geq 0$ and $\int_{\re} \omega(s) ds=1$
and $\kappa_\epsilon(x,y):= \frac{1}{\epsilon^n} \kappa \left(\frac{d_g(x,y)}{\epsilon}\right)$
with $\kappa  \in C^\infty(\re)$, $\text{supp}(\kappa) \subset (-1,1)$, $\kappa \geq 0$ and  $\int_{\re^n} \kappa(|z|) dz=1$.
Thus, (\ref{eqEind1}) yields
\begin{equation*}
\begin{aligned}
& \int_0^T \int_M\psi'(t) \bar\psi(x)  \int_0^T\int_M \omega_{\bar{\epsilon}}(t-s) \kappa_\epsilon(x,y)|u(x,t)-v(y,s)|\, dv_g(y)\,ds\,  dv_g(x)\, dt  \\
&+  \int_0^T\int_M \psi(t)  \int_0^T\int_M  \omega_{\bar{\epsilon}}(t-s)
\left[ \langle q(u(x,t), v(y,s),x,t),\grad_g^x \kappa_\epsilon (x,y) \rangle_g \bar\psi(x)\right. \\ 
&+  \langle q(u(x,t), v(y,s),x,t),\grad_g^x \bar\psi(x) \rangle_g  \kappa_\epsilon (x,y) \\
&+ \left.\langle q(v(y,s),u(x,t),y,s),\grad_g^y \kappa_\epsilon (x,y) \rangle_g \bar\psi(x)\right]\, 
dv_g(y) \,ds\,  dv_g(x) \, dt   \geq 0. 
\end{aligned}
\end{equation*}
% We define
% \begin{align*}
% I_1^{\epsilon, \bar{\epsilon}}(x,t):=& \int_M \int_0^T \omega_{\bar{\epsilon}}(t-s) \kappa_\epsilon(x,y)|u(x,t)-v(y,s)|\,  ds\,  dv_g(y),\\
% I_2^{\epsilon, \bar{\epsilon}}(x,t):=& \int_M \int_0^T  \omega_{\bar{\epsilon}}(t-s) \\
% &\left[ \langle q(u(x,t), v(y,s),x,t),\grad_g^x \kappa_\epsilon (x,y) \rangle_g \bar\psi(x)\right.  \\ 
% &+ \left.\langle q(v(y,s),u(x,t),y,s),\grad_g^y \kappa_\epsilon (x,y) \rangle_g \bar\psi(x)\right]\, 
% ds\,  dv_g(y),\\
% I_3^{\epsilon, \bar{\epsilon}}(x,t):=&\int_M \int_0^T  \omega_{\bar{\epsilon}}(t-s)
%  \langle q(u(x,t), v(y,s),x,t),\grad_g^x \bar\psi(x) \rangle_g  \kappa_\epsilon (x,y) \,ds\,  dv_g(y).
% \end{align*}
With the same argumentation as in \cite[1714-1718]{MuellerLengeler} we obtain by subsequently letting $\bar \epsilon$ and 
$\epsilon$ tend to zero

\begin{equation*}
\begin{aligned}
\int_0^T \int_M\psi' \bar \psi & |u-v|+ \psi\, \sgn(u-v)\, 
\langle \grad_g \bar \psi , f(u)-f(v) \rangle_g dv_g \, dt \geq 0.
\end{aligned}
\end{equation*}
% \begin{equation*}
% \begin{aligned}
% \int_0^T \int_M\psi'(t) \bar \psi (x)& |u(x,t)-v(x,t)|+ \psi(t)\, \sgn(u(x,t)-v(x,t))\, \\
% &\langle \grad_g \bar \psi (x), f(u(x,t),x,t)-f(v(x,t),x,t) \rangle_g dv_g \, dt \geq 0
% \end{aligned}
% \end{equation*}
%for every $\psi\in C_0^\infty ((0,T))$ with $\psi \geq 0$ and $\bar \psi  \in C_0^\infty(M)$ with $\bar \psi \geq 0$,
%where we suppressed the $(x,t)$-dependence.
%We do no longer mention the $(x,t)$-dependence. 
Setting $\bar \psi:= 1-  R_\delta$ we get
% \begin{equation*}
% \begin{aligned}
%  &\int_0^T \int_M \psi' \, (1-R_\delta)\,  |u-v|  
% -\psi\, \sgn(u-v)\,  \langle \grad_g R_\delta, f(u)-f(v) \rangle_g
% \, dv_g \, dt \geq 0.
% \end{aligned}
% \end{equation*}
with Lemma \ref{lemSpurFluss} for $\delta \searrow 0$
 \begin{equation}
 \begin{aligned}
\int_M \int_0^T \psi' | u-v|\,dt\, dv_g\geq
  \int_{\partial M} \int_0^T \psi \sgn(Tu- Tv) 
 \left\langle f(Tu)-f(Tv), N\right\rangle_gdt\,  dv_{\widetilde g}.
\end{aligned}
\label{RSLS}
\end{equation}
Here, we used that $\sgn(u-v)(f(u)-f(v))= f(\max\{u,v\}) - f(\min\{u,v\})$ and the fact that $\max\{u,v\}$ and $\min\{u,v\}$ inherit a
bounded variation from $u$ and $v$.
In the next lines we will show that the right-hand side of (\ref{RSLS}) is nonnegative. With
\begin{align*}
 \bar k := \begin{cases}
           Tu\, &\text{ if $Tu \in I(0,Tv)$} \\
           0 \, &\text{ if $0 \in I(Tu, Tv)$ }\\
           Tv\, &\text{ if $Tv \in I(0,Tu)$},
          \end{cases}
\end{align*}
where $I(a,b):= [\min\{a,b\}, \max\{a,b\}]$ we obtain
\begin{align*}
 &\int_{\partial M} \int_0^T \psi(t)\, \sgn(Tu- Tv) \, \left\langle f(Tu)-f(Tv), N\right\rangle_g\,dt\,  dv_{\widetilde g}\\
 =&\int_{\partial M} \int_0^T \psi(t)\, \sgn(Tu- \bar k ) \, \left\langle f(Tu)-f(\bar k), N\right\rangle_g\,dt\,  dv_{\widetilde g}\\
 &+\int_{\partial M} \int_0^T \psi(t)\, \sgn(Tv- \bar k ) \, \left\langle f(Tv)-f(\bar k), N\right\rangle_g\,dt\,  dv_{\widetilde g}.\\
\end{align*}
In order to show that each summand is nonnegative we exploit inequality 
(\ref{entrol}) with $\phi= \bar \phi\, R_\delta$ for $\bar \phi \in C_0^\infty (\bar M \times (0,T))$, 
$ \bar\phi(x,t)\geq 0$  and obtain
\begin{align*}
&\int_M \int_0^T |u-k|\, \partial_t\bar \phi \, R_\delta 
+\sgn (u-k)\left\langle f(u)-f(k),\grad_g R_\delta \right\rangle_g  \bar \phi
\, dt\, dv_g\\
&+\sgn (u-k)\left\langle f(u)-f(k),\grad_g \bar \phi \right\rangle_g R_\delta
\, dt\, dv_g\\
&+\int_{\partial M} \int_0^T \sgn (k) \left\langle f(Tu)-f(k),N\right \rangle_{g}\, \bar \phi\, dt \, dv_{\widetilde g} \geq 0
\end{align*}
for all $k \in \re$. By Lemma \ref{lemSpurFluss} we have for $\delta \searrow 0$
\begin{align*}
 \int_{\partial M} \int_0^T \left(\sgn (Tu-k) +\sgn (k)\right)
 \left\langle f(Tu)-f(k),N\right\rangle_{g}\, \bar \phi\,  dt \, dv_{\widetilde g}
 \geq 0
\end{align*}
for all $k \in \re$. Since $\bar \phi$ was arbitrary, obviously 
\begin{align*}
 \left(\sgn (Tu-k) +\sgn (k)\right)
 \left\langle f(Tu)-f(k),N\right\rangle_{g}\geq 0
\end{align*}
almost everywhere on $\partial M\times(0,T)$. 
Using the fact that 
\begin{align*}
 \sgn(Tu-\bar k) = \begin{cases}
           0\, &\text{ if } \bar k = Tu                             ,   \\
           \sgn(Tu-\bar k) + \sgn(\bar k)  \, &\text{ if } \bar k=0 , \\
           \frac{1}{2}(\sgn(Tu-\bar k) + \sgn(\bar k)) \, &\text{ if }\bar k=Tv,
          \end{cases}
\end{align*}
we see, after a repetition of the argumentation for $v$, that the right-hand side of (\ref{RSLS}) is nonnegative and consequently
\begin{align}
  \int_M \int_0^T \psi'(t)\, | u-v|\, dt\, dv_g \geq 0.\label{LSt}
 \end{align}
 Let $\Psi$ denote the characteristic function of an arbitrary time interval $[t_0, t_1] \subset (0,T)$ 
 and $\psi_\epsilon= \Psi \ast \omega_\epsilon$ its mollification.
 For $\psi= \psi_\epsilon$ in \eqref{LSt} we obtain as $\epsilon$ tends to zero
%  \begin{align*}
%  0 &\leq \lim_{\epsilon \searrow 0} \int_M \int_0^T |u-v|\,  \psi_\epsilon'\, dt\, dv_g \\
%  &= | u(\cdot ,t_0)- v(\cdot ,t_0)|_{L^1(M)} -| u(\cdot,t_1)- v(\cdot,t_1)|_{L^1(M)}\\
% \end{align*}
% and thus
\begin{align}\label{eq:contraction}
 \| u(\cdot,t_1)- v(\cdot,t_1)\|_{L^1(M)} \leq \| u(\cdot,t_0)- v(\cdot,t_0)\|_{L^1(M)}. 
\end{align}
Letting $t_0$ tend to zero we obtain uniqueness.
\end{proof}
\begin{corollary}[$L^1$-contraction property]
 Let $u,v$ be two entropy solutions of problem (\ref{11}), (\ref{12}), (\ref{boundarycondition}), then \eqref{eq:contraction}
%  \begin{equation*}
%  \| u(\cdot,t_1)- v(\cdot,t_1)\|_{L^1(M)} \leq \| u(\cdot,t_0)- v(\cdot,t_0)\|_{L^1(M)}. 
%  \end{equation*}
holds for $0\leq t_0\leq t_1$.
\end{corollary}

%\printbibliography
\bibliographystyle{plain}      % <---------------------------------- MOD
\bibliography{bibliography}

\end{document}